\documentclass[a4paper,11pt]{amsart}

\usepackage{amsmath} 
\usepackage{amssymb} 
\usepackage{amsthm} 
\usepackage{mathrsfs} 
\usepackage{mathtools} 
\usepackage{aliascnt} 
\usepackage{float} 
\usepackage{tikz} 
\usepackage{pgfplots} 
\usepackage{hyperref} 
\usepackage{nameref} 
\usepackage{todonotes}
\usepackage{xspace} 
\usepackage[ 
    nameinlink 
]{cleveref} 
\usepackage{autonum} 
\usepackage{titlesec} 
\usepackage{microtype}
\usepackage{enumitem}

\titlelabel{\thetitle.\quad}

\hypersetup{
    colorlinks=true,
    linkcolor=blue,
    citecolor=orange,
    filecolor=magenta,      
    urlcolor=cyan,
    }
\pgfplotsset{width=8cm,compat=1.9}

\newtheorem{Theorem}{Theorem}[section]

\newaliascnt{Lemma}{Theorem}

\aliascntresetthe{Lemma}

\newaliascnt{Proposition}{Theorem}
\newtheorem{Proposition}[Proposition]{Proposition}
\aliascntresetthe{Proposition}

\newaliascnt{Corollary}{Theorem}
\newtheorem{Corollary}[Corollary]{Corollary}
\aliascntresetthe{Corollary}

\newaliascnt{Definition}{Theorem}
\newtheorem{Definition}[Definition]{Definition}
\aliascntresetthe{Definition}

\newaliascnt{Remark}{Theorem}
\newtheorem{Remark}[Remark]{Remark}
\aliascntresetthe{Remark}

\newaliascnt{Example}{Theorem}
\newtheorem{Example}[Example]{Example}
\aliascntresetthe{Example}

\crefname{Theorem}{Theorem}{Theorems}
\Crefname{Theorem}{Theorem}{Theorems}

\crefname{Lemma}{Lemma}{Lemmas}
\Crefname{Lemma}{Lemma}{Lemmas}

\crefname{Proposition}{Proposition}{Propositions}
\Crefname{Proposition}{Proposition}{Propositions}

\crefname{Corollary}{Corollary}{Corollaries}
\Crefname{Corollary}{Corollary}{Corollaries}

\crefname{Definition}{Definition}{Definitions}
\Crefname{Definition}{Definition}{Definitions}

\crefname{Remark}{Remark}{Remarks}
\Crefname{Remark}{Remark}{Remarks}

\crefname{Example}{Example}{Examples}
\Crefname{Example}{Example}{Examples}

\calclayout

\newcommand{\RR}{\mathbb{R}}
\newcommand{\NN}{\mathbb{N}}
\newcommand{\cont}{\mathcal{C}}

\usepackage{scalerel}[2014/03/10]
\usepackage[usestackEOL]{stackengine}

\def\dashint{\,\ThisStyle{\ensurestackMath{%
  \stackinset{c}{.2\LMpt}{c}{.5\LMpt}{\SavedStyle-}{\SavedStyle\phantom{\int}}}%
  \setbox0=\hbox{$\SavedStyle\int\,$}\kern-\wd0}\int}
\def\ddashint{\,\ThisStyle{\ensurestackMath{%
  \stackinset{c}{.2\LMpt}{c}{.5\LMpt+.2\LMex}{\SavedStyle-}{%
    \stackinset{c}{.2\LMpt}{c}{.5\LMpt-.2\LMex}{\SavedStyle-}{%
      \SavedStyle\phantom{\int}}}}\setbox0=\hbox{$\SavedStyle\int\,$}\kern-\wd0}\int}

\makeatletter
\renewcommand{\l@section}{\@tocline{1}{0pt}{1.5pc}{3pc}{}}
\makeatother

\begin{document}
\title[Viscosity Framework for Dynamic Programming Principles]{\textbf{A Viscosity Framework for Dynamic Programming Principles and  Applications}}
\date{}

\author[F.~del~Teso]{F\'elix del Teso}

\address[F. del Teso]{Departamento de Matematicas, Universidad Aut\'onoma de Madrid (UAM), Campus de Cantoblanco, 28049 Madrid, Spain} 
\email[]{felix.delteso\@@{}uam.es}

\urladdr{https://sites.google.com/view/felixdelteso}


\author[J. D. Rossi]{Julio D. Rossi}
\address[J. D. Rossi]{Departamento de Matem{\'a}tica y Estadistica, 
        Universidad Torcuato di Tella,
         Av. Figueroa Alcorta 7350, (1428),
        Buenos Aires, Argentina.}
\email{julio.rossi@utdt.edu}

\author[J. Ruiz-Cases]{Jorge Ruiz-Cases}
\address[J.\,Ruiz-Cases]{Departamento de Matem\'{a}ticas, Universidad Aut\'onoma de Madrid,
    \& Instituto de Ciencias Matem\'{a}ticas ICMAT (CSIC-UAM-UCM-UC3M),
    28049-Madrid, Spain.}
\email{jorge.ruizc@uam.es}
\urladdr{https://sites.google.com/view/jorgeruizcases}

      

\begin{abstract}
In this work we introduce a viscosity-based notion of solution for general approximation schemes associated with partial differential equations, such as dynamic programming principles~(DPPs). A key feature of our approach is that it bypasses any measurability requirement on solutions of the DPP, an assumption that is often difficult to verify and may even fail in relevant examples. We establish a comparison principle between classical strict supersolutions and viscosity subsolutions of the DPP, which yields stability results under minimal and natural hypotheses. As a consequence, we prove existence of viscosity solutions of the DPP and their convergence to viscosity solutions of a PDE that is consistent with the underlying approximation scheme. Moreover, we show that solutions of the limiting PDE admit an asymptotic expansion encoded by the approximation operator. Finally, we demonstrate that a broad class of local, nonlocal, and nonlinear partial differential equations fits into our framework, recovering known examples in the literature and completing gaps in the existing literature.
\color{purple}

\normalcolor
\end{abstract}

\maketitle

\tableofcontents

\section{Introduction} \label{se:introduction}

We consider a family of operators~$\mathfrak{a}_\rho$, typically of average-type and possibly arising from an approximation of a differential operator. A classical example is provided by the asymptotic expansion of the Laplacian: for any smooth function~$\varphi$,
\[
-\Delta \varphi(x)
= \frac{2(N+2)}{\rho}\left(\varphi(x)-\mathfrak{a}_\rho(x,\varphi)\right) + o_\rho(1), \quad \textup{where} \quad \mathfrak{a}_\rho(x,\varphi)=\dashint_{B_{\sqrt{\rho}}(x)} \varphi(y)\,dy.
\]
The operator~$\mathcal{A}_\rho$  defined by
\begin{equation}
    \mathcal{A}_\rho(x,\varphi,\varphi(x)) \coloneqq \frac{2(N+2)}{\rho}\left(\varphi(x)-\mathfrak{a}_\rho(x,\varphi)\right)
\end{equation}
is usually referred in the literature as a \emph{consistent} approximation of~$-\Delta$ in the sense that
\[
-\Delta \varphi(x) = \mathcal{A}_\rho(x,\varphi,\varphi(x)+o(\rho)).
\]
More generally, let~$F$ be a second--order elliptic operator and assume that~$\mathcal{A}_\rho$ provides a consistent approximation of~$F$ in the sense that for every smooth test function~$\varphi$,
\[
F\bigl(x,\varphi(x),\nabla \varphi(x),D^2\varphi(x)\bigr)
= \mathcal{A}_\rho(x,\varphi,\varphi(x)+o(\rho)) ,
\quad \textup{with} \quad
\mathcal{A}_\rho(x,\varphi,\varphi(x))
= \frac{c}{\rho}\left(\varphi(x)-\mathfrak{a}_\rho(x,\varphi)\right).
\]

It has been systematically observed in the literature that the aforementioned consistency property closely relates the following two statements:

\begin{enumerate}[
  label=(\Roman*),
  leftmargin=0pt,
  labelindent=0pt,
  labelwidth=!,
  align=left
]
\item\label{item2} \emph{Approximation schemes.}
The family of solutions~$\{u_\rho\}_{\rho>0}$ of
\begin{equation}\label{eq:DPPintro}
\begin{cases}
\mathcal{A}_\rho(x,u_\rho,u_\rho(x))=0, & x \in \Omega,\\
u_\rho(x)=g(x), & x \in \mathbb{R}^N\setminus\Omega,
\end{cases}
\quad
\left(\textup{or,}\quad
\begin{cases}
u_\rho(x)=\mathfrak{a}_\rho(x,u_\rho), & x \in \Omega,\\
u_\rho(x)=g(x), & x \in \mathbb{R}^N\setminus\Omega,
\end{cases}\right)
\end{equation}
converges uniformly, as~$\rho\to0^+$, to the unique viscosity solution of
\begin{equation}\label{eq:BVPintro}
\begin{cases}
F\bigl(x,u(x),\nabla u(x),D^2u(x)\bigr)=0, & x \in \Omega,\\
u(x)=g(x), & x \in \mathbb{R}^N\setminus\Omega.
\end{cases}
\end{equation}
Here we put a greater emphasis on solutions of the approximation scheme, and how those are related to solutions of the equation. This property is typically known as ``convergence'' in the context of numerical analysis.

\item\label{item1} \emph{Asymptotic mean value property.}
A function~$u$ is a solution of
\begin{equation}\label{eq:PDEintro}
F\bigl(x,u(x),\nabla u(x),D^2u(x)\bigr)=0 \quad \textup{in } \Omega
\end{equation}
if and only if it satisfies the asymptotic mean value formula,
\begin{equation}\label{eq:AMVPintro}
\mathcal{A}_\rho(x,u,u(x)+o(\rho))=0,
\quad
\left(\textup{or,}\quad
u(x)=\mathfrak{a}_\rho(x,u)+o(\rho)\right),
\qquad \textup{as} \quad \rho\to0^+ \quad \textup{for all} \quad x\in\Omega.
\end{equation}
Solutions to both the PDE and the asymptotic mean value formula are understood in the viscosity sense. Here the emphasis is on the solutions the PDE, and how they are almost solutions to the approximation scheme. This property is sometimes referred as ``local truncation error'' in the context of numerical analysis.

\end{enumerate}

\begin{Remark}
    Problem~\eqref{eq:DPPintro} provides a unifying framework for game theory, optimal control and numerical analysis. In game theory, they naturally arise as an intermediate step between the game and the PDEs, as they characterize the value of the game. In optimal control, they encode optimal decision-making over time. And in numerical analysis, DPPs work as approximation schemes that converge to the continuous  solution of the PDE. Following the terminology used in some of these areas, we will refer to problem~\eqref{eq:DPPintro} as \emph{Dynamic Programming Principle} (DPP).
\end{Remark}

\medskip

\noindent \textbf{Relevance of problems {\normalfont~\ref{item2}} and {\normalfont~\ref{item1}}.} 
The relevance of both problems~\ref{item2} and~\ref{item1} is well established; nevertheless, we briefly recall them for the reader’s convenience. Concerning problem~\ref{item2}, it is often significantly easier to analyze solutions of the DPP than to deal directly with solutions of the corresponding PDE. By passing to the limit in the approximation scheme, solutions of the DPP can be used not only to establish existence results for the PDE, but also to infer qualitative properties that may be considerably more difficult to obtain through purely PDE-based arguments.  

Regarding problem~\ref{item1}, characterizing solutions of the PDE through an asymptotic mean value property often provides valuable structural insight into the nature of the solutions. Among other consequences, since the operator~$\mathcal{A}_\rho$ typically does not impose differentiability requirements, this characterization allows one to extend the notion of solution of the PDE to more general functional settings, where classical differentiability may not be available.

The present paper has \textbf{four main goals} related to problems~\ref{item2} and~\ref{item1}.

\medskip
\noindent\textbf{Objective 1.}
One of the main challenges in the study of dynamic programming principles is to prove the existence of solutions to~\eqref{eq:DPPintro}. A typical approach consists in looking for a fixed point of the operator
\[
T[u](x) \coloneqq \mathfrak{a}_\rho(x,u).
\]
For this purpose, one needs to identify a suitable functional space~$\mathcal{X}$ such that~$T\colon \mathcal{X}\to\mathcal{X}$. Since the operator~$T$ usually involves some form of integral averaging, the space~$\mathcal{X}$ is often required to be contained in the set of measurable functions. However, this requirement may fail to hold even in relatively simple examples. Indeed, there exist DPPs (associated with meaningful PDEs) for which no existence theory of classical solutions is available precisely due to measurability issues.

A simple representative case is given by the averaging operator
\begin{equation}\label{eq:nonmeaseaxample}
\mathfrak{a}_\rho(x,u)
= \alpha \dashint_{B_{\sqrt{\rho}}(x)} u(y)\,dy
+ (1-\alpha) \sup_{y\in \overline{B}_\rho(x)} u(y),
\end{equation}
with~$\alpha\in(0,1)$, which arises naturally in the study of the equation~$-\Delta u - |\nabla u| = 0$.
While the integral term clearly maps measurable functions into measurable functions, it was shown in~\cite{LPS} that the supremum term
\(
\sup_{y\in \overline{B}_\rho(x)} u(y)
\)
may fail to be Borel measurable, even when~$u$ itself is Borel measurable. 

Another prominent and more recent example is provided by the averaging operator associated to the infinity fractional Laplacian, see~\cite{[BCFinfinity]}. There, the approximation operator found in~\cite{delTesoEndalLewicka} involves directional suprema and infima along one-dimensional integrals on rays.
The question of existence of solutions to the DPP is open since it is not known whether this averaging operator maps measurable functions to measurable functions. Using the theory developed in this paper, we solve the existence problem in~\Cref{se:examples}, Example~\ref{ej:inflapl}.

The present work is motivated by the observation that these measurability requirements constitute a genuine obstruction, rather than a merely technical inconvenience. Avoiding such complications is a main goal of this paper. To overcome this difficulty, we introduce a notion of \emph{viscosity solution} to the DPP~\eqref{eq:DPPintro} that completely avoids measurability issues. Even in situations where measurability can be guaranteed, verifying it may require delicate and technically involved arguments.

An additional motivation for introducing a viscosity-type notion of solutions for the DPP is that it completes the viscosity framework underlying problems~\ref{item2} and~\ref{item1}. Indeed, the boundary value problem~\eqref{eq:BVPintro} appearing in problem~\ref{item2}, as well as the PDE~\eqref{eq:PDEintro}, and the asymptotic mean value formula~\eqref{eq:AMVPintro} in problem~\ref{item1} are all interpreted in the viscosity sense. It is therefore natural and conceptually consistent to extend the notion of viscosity solutions to the DPP~\eqref{eq:DPPintro} as well.

Our new notion of solution for~\eqref{eq:DPPintro} is inspired by the classical theory of viscosity solutions for elliptic PDEs: sub- and supersolutions are defined using smooth test functions, and a solution is understood as a function that is both a subsolution and a supersolution.

The main strengths of this approach can be summarized as follows:
\begin{itemize}
\item It eliminates the need to verify measurability (or any other structural property of the solution space) altogether.
\item It is naturally suited to Perron-type arguments and is designed to guarantee existence of solutions under minimal assumptions on~$\mathcal{A}_\rho$ (or, equivalently, on~$\mathfrak{a}_\rho$).
\item Just as in the PDE case, there is an easy-to-check comparison result, under minimal hypotheses, between a classical strict supersolution and a viscosity subsolution (and vice versa) of the DPP.
\end{itemize}

The present work can be regarded as a follow-up to~\cite{LiuSchikorra}. There, the authors also provide a systematic approach to solving explicit DPPs in the classical sense. However, their approach relies on the hypothesis that the pointwise supremum of \emph{nice} subsolutions remains \emph{nice}, where \emph{nice} could mean, for instance, \emph{measurable}. We consider this to be a very strong assumption, which our method completely bypasses.

\medskip

\noindent\textbf{Objective 2.} One of the hypotheses required both for the previous existence result for the DPP and for the convergence result toward solutions of the PDE (Objective 3) is that of \emph{stability}: namely, we require the existence of an upper bound for any viscosity subsolution (resp.\@ a lower bound for any viscosity supersolution)~$u_\rho$ of the DPP.

In the context of the DPP, these bounds may depend on~$\rho$; however, in order to obtain convergence to the PDE, they must be uniform with respect to~$\rho$. We address this issue by means of the comparison result for the DPP mentioned above, under the assumption of the existence of a classical strict supersolution (resp.\@ subsolution).

In order to check stability properties of the DPP, we provide a systematic approach to constructing such strict supersolutions, based on smooth solutions of the limiting PDE together with a consistency property of the scheme. This technique is useful to establish the hypotheses of both Objective~1 and also in the forthcoming Objective~3.

\normalcolor

\medskip

\noindent\textbf{Objective 3.} 
Here, our aim is to study the convergence of viscosity solutions~$u_\rho$ of the DPP toward solutions of the PDE. The tools developed above allow us to recover property~\ref{item2} within the framework of viscosity solutions of the DPP~\eqref{eq:DPPintro}, under weaker assumptions than those traditionally imposed.

In particular, this result can be viewed as a refinement of the classical work of Barles and Souganidis~\cite{BarlesSouganidis}, where the existence of classical solutions of~\eqref{eq:DPPintro} was required in order to guarantee convergence of the approximation scheme. Moreover, by exploiting the method described in Objective~2, we are able to dispense with the stability assumption altogether.

\medskip

\noindent\textbf{Objective 4.} Our final goal is to present a general and systematic treatment of the equivalence between being a solution to the PDE
\begin{equation}
F\bigl(x,u(x),\nabla u(x),D^2u(x)\bigr)=0,
\end{equation}
and being a solution to the asymptotic mean value property
\begin{equation} \label{eq:intuitivemeanvalue}
    \mathcal{A}_\rho(x,u,u(x)+ o(\rho))=0, 
    \quad \left( \textup{or} \quad u(x)=\mathfrak{a}_\rho(x,u)-o(\rho)\right),
    \quad \textup{as } \rho\to0^+,
\end{equation}
both notions being understood in the viscosity sense. We define viscosity solutions for the asymptotic mean value formula giving a rigorous interpretation of formula~\eqref{eq:intuitivemeanvalue} in the weakest natural way.

Adopting this interpretation of viscosity solutions for the mean value formula allows us to prove that the following consistency condition
\begin{equation} \label{eq:consistencia}
    \lim_{\rho\to0^+} \mathcal{A}_\rho(x,\varphi,\varphi(x)+ o(\rho))
    = F(x, \varphi(x), \nabla \varphi(x), D^2\varphi(x))
    \quad \textup{for all } \varphi \in C_b^\infty(\RR^N).
\end{equation}
In fact a weaker version of it, see~\Cref{teo.asympt.intro}, already yields the equivalence between viscosity solutions of the PDE and viscosity solutions of the asymptotic mean value formula, essentially by definition.

\subsection{Comments on related literature.}

Mean value formulas for linear operators beyond the Laplacian have been studied in the literature. 
Littman et al.~\cite{Littman.et.al1963} proved a mean value theorem for linear divergence-form operators with bounded measurable coefficients. 
Caffarelli~\cite{Caf98} (see also~\cite{BH2015,Caffarelli.Roquejoffre.2007}) highlighted a related result in terms of mean value sets.

They are also widely known for nonlinear operators. The most classical example in this setting is the normalized (also called~$1$-homogeneous or game-theoretic)~$p$-Laplacian, see~\cite{AsymptoticMeanValuePharmonic,PSSW2}. For historical context and further reading, we refer to the works 
\cite{AHP,ArroyoLLorente, AsymptoticMeanValueFormulas, [Blanc et al. 2020], BChMR, BS, CIMW, delTesoLindgrenMeanValue, DelTeso-Manfredi.Parviainen, DTR, DMRS, FLM, GMRR, Hartikainen, AnotherapproachPLaplacian,  KohnSerfatyCurvature, KohnSerfatyFullyNonlinear, LM, LiMa, LiuSchikorra, LPS, AsymptoticMeanValueTugofwar, AsymptoticMeanValuePharmonic}, 
as well as to the recently published books~\cite{[Blanc and Rossi 2019], Lew}. These works cover a wide range of nonlinear equations, including Monge–Amp\`ere equations,~$k$-Hessians, Bellman equations, geometric flows, and more, for which mean value formulas satisfying properties~\ref{item2} and~\ref{item1} are available.

It is also worth mentioning that there are cases in which the asymptotic mean value formula~\eqref{eq:AMVPintro} (and also others) holds pointwise. Examples can be found in~\cite{ArroyoLLorente, delTesoLindgrenMeanValue, LiMa}.
\normalcolor

\subsection{Structure of the paper.}
In~\Cref{se:defandres}, we introduce the main definitions and state our main results. 
\Cref{se:existencia} is devoted to the proof of existence of viscosity solutions to the DPP. 
In~\Cref{se:comparison}, we establish a comparison principle and discuss its role within the overall framework.~\Cref{se:BS} addresses the convergence of solutions of the DPP to solutions of the limiting PDE. 
In~\Cref{se:asymptotic}, we prove the equivalence between solutions of the asymptotic mean value formula and solutions of the PDE. 
Finally,~\Cref{se:extensions} presents several possible extensions of the theory, while~\Cref{se:examples} illustrates the scope of our results by showing that our theory covers and extends most of the known examples in the literature.

\section{Definitions and main results} \label{se:defandres}

 Consider~$\overline{\mathcal{X}}$ to be set of \emph{bounded} functions from~$\RR^N$ to~$\RR$, that is
\[
\overline{\mathcal{X}} \coloneqq \{f: \RR^N \to \RR \text{ such that } \sup_{x \in \Omega} \{|f(x)|\} < \infty \},
\]
and let~$\mathcal{X}$ be a fixed subset of~$\overline{\mathcal{X}}$. 

\begin{Remark}
Throughout the paper, and in view of the previous examples, one may think of~$\mathcal{X}$ as the set of \emph{bounded measurable} functions from~$\RR^N$ to~$\RR$. Moreover, the extra inclusion~$\cont_b^\infty(\Omega) \subset \mathcal{X}$ will sometimes be required. For example, it will be useful when dealing with test functions in the study of the convergence of DPP solutions towards the solution of the associated PDE.
\end{Remark}

\begin{Remark} \label{obs:RN}
Using~$\RR^N$ in the definitions of~$\mathcal{X}$ and~$\overline{\mathcal{X}}$ plays no role in when discussing a viscosity solution of the DPPs, and could be easily change to any other set. However, we keep~$\RR^N$ for simplicity.
\end{Remark}

Given a function~$\mathcal{A}: \RR^N \times \mathcal{X} \times \RR \to \RR$ we consider the associated implicit DPP through the following definition.
\begin{Definition}
    Let~$\mathcal{A}: \Omega \times \mathcal{X} \times \RR \to \RR$. We say that~$u\in \mathcal{X}$ is a \emph{classical solution} of associated implicit DPP inside an open domain~$\Omega$ if
   ~$$
    \mathcal{A}(x, u, u(x)) = 0\quad \textup{for all} \quad x\in\Omega.
   ~$$
\end{Definition}

\begin{Remark}
The terminology ``implicit'' in the previous definition is used to emphasize that, in general, it may not be possible to identify an explicit averaging operator~$\mathfrak{a}$ such that~$u(x) = \mathfrak{a}(x,u)$.
This is the case, for instance, for the implicit DPP introduced in~\cite{delTesoLindgrenMeanValue}, which is valid for the~$p$-Laplacian operator~$\mathrm{div}\bigl(|\nabla u|^{p-2}\nabla u\bigr)$,
and is given by
\begin{equation}
\mathcal{A}_\rho(x,u,u(x))
=
\frac{C}{\rho^p}
\dashint_{B_\rho}
|u(x+y)-u(x)|^{p-2}\bigl(u(x+y)-u(x)\bigr)\, dy
\end{equation}
for some constant~$C>0$ depending only on~$N$ and~$p$. 
\end{Remark}

The existence of an associated averaging operator~$\mathfrak{a}$ can nevertheless be guaranteed under very mild assumptions on~$\mathcal{A}$, which are the following:
\begin{enumerate}[label={(\alph*)}]
    \item \label{asA-it:a}
   ~$\mathcal{A}$ is nonincreasing in the second variable; that is, for every~$(x,s)\in \Omega \times\RR$ and every~$\varphi_1,\varphi_2\in \mathcal{X}$ such that~$\varphi_1\leq \varphi_2$ in~$\RR^N$, we have
    \[
    \mathcal{A}(x, \varphi_2, s)\leq \mathcal{A}(x, \varphi_1, s).
    \]

    \item \label{asA-it:abis}
   ~$\mathcal{A}$ is nondecreasing in the third variable; that is, for every~$(x,\varphi) \in \Omega \times \mathcal{X}$ and every~$s_1,s_2\in \RR$ such that~$s_1\leq s_2$, we have
    \[
    \mathcal{A}(x, \varphi, s_1)\leq \mathcal{A}(x, \varphi, s_2).
    \]

    \item \label{asA-it:b}
    For every~$(x,\varphi) \in \RR^N \times \mathcal{X}$, the function~$\theta\colon \RR \to \RR$ defined by~$\theta(s) \coloneqq \mathcal{A}(x,\varphi,s)$
    admits a unique zero; that is, there exists a unique~$s_0 \in \RR$ such that~$\theta(s_0)=0$.
\end{enumerate}

\begin{Remark}
If in assumption~\textup{\ref{asA-it:abis}} we additionally assume that~$\mathcal{A}$ is \emph{strictly} increasing in the third variable, then in~\textup{\ref{asA-it:b}} it would be sufficient to require existence of a solution~$s$ to~$\mathcal{A}(x,\varphi,s)=0$,
since uniqueness would then follow from~\textup{\ref{asA-it:abis}}.

\end{Remark}

\begin{Remark}\label{rem:defAa}
Observe that, thanks to assumption~\textup{\ref{asA-it:b}} on~$\mathcal{A}$, we can define~$\mathfrak{a} \colon \RR^N \times \mathcal{X} \to \RR$
by
\begin{equation}
    \mathfrak{a}(x,\varphi) := s_{x,\varphi},
\end{equation}
where~$s_{x,\varphi}$ is the unique value such that~$\mathcal{A}(x,\varphi,s_{x,\varphi}) = 0$. This gives rise to an equivalent ``explicit'' DPP, namely,
\[
\mathcal{A}(x,u,u(x)) = 0
\quad \Longleftrightarrow \quad
u(x) = \mathfrak{a}(x,u).
\]

Note that, thanks to the monotonicity assumptions~\textup{\ref{asA-it:a}} and~\textup{\ref{asA-it:abis}} on~$\mathcal{A}$, the operator~$\mathfrak{a}$ is nondecreasing in its second variable.  
Indeed, if~$\varphi_1 \leq \varphi_2$, then
\[
\mathcal{A}(x,\varphi_2,s_{x,\varphi_1})
\leq
\mathcal{A}(x,\varphi_1,s_{x,\varphi_1})
=
0
=
\mathcal{A}(x,\varphi_2,s_{x,\varphi_2}),
\]
and by the monotonicity of~$\mathcal{A}$ with respect to the third variable we conclude that
\[
\mathfrak{a}(x,\varphi_1)
=
s_{x,\varphi_1}
\leq
s_{x,\varphi_2}
=
\mathfrak{a}(x,\varphi_2).
\]
\end{Remark}

\begin{Remark}
Conversely, if a mapping~$\mathfrak{a} \colon \Omega \times \mathcal{X} \to \RR$ that is nondecreasing in the second variable is given, we can define
\[
\mathcal{A} \colon \Omega \times \mathcal{X} \times \RR \to \RR
\]
by
\begin{equation}
    \mathcal{A}(x,\varphi,s) := s - \mathfrak{a}(x,\varphi).
\end{equation}
Observe that~$\mathcal{A}$ trivially satisfies assumptions~\textup{\ref{asA-it:a}--\ref{asA-it:abis}--\ref{asA-it:b}}. However, this is not the only way to define~$\mathcal{A}$ from~$\mathfrak{a}$, and sometimes there may be other convenient choices, depending on the consistency assumptions. We postpone this discussion to the end of~\Cref{se:BS}.
\end{Remark}

We can now define classical solutions to the following DPP with boundary values:
\begin{equation}\label{eq:BVPDPP}
\begin{cases}
\mathcal{A}(x,u,u(x)) = 0, & x \in \Omega, \\
u(x) = g(x), & x \in \RR^N \setminus \Omega,
\end{cases}
\qquad
\left(\textup{or equivalently,} \quad
\begin{cases}
u(x) = \mathfrak{a}(x,u), & x \in \Omega, \\
u(x) = g(x), & x \in \RR^N \setminus \Omega.
\end{cases}
\right)
\end{equation}

\begin{Definition}[Classical solution of the DPP with boundary values]\label{def.clasica}
Let~$\Omega \subset \RR^N$ be an open domain, let~$g:\RR^N\setminus \Omega \to \RR$, and let
$\mathcal{A}\colon \Omega \times \mathcal{X} \times \RR \to \RR$
satisfy assumptions~\textup{\ref{asA-it:a}--\ref{asA-it:abis}--\ref{asA-it:b}}. Let~$\mathfrak{a}:   \RR^N \times \mathcal{X} \to \RR$ be defined as in~\Cref{rem:defAa}.

We say that~$u \in \mathcal{X}$ is a \emph{classical supersolution} of~\eqref{eq:BVPDPP} if
$u(x) \geq g(x)$ for all ~$x \in \RR^N \setminus \Omega$,
and
\[
\mathcal{A}(x,u,u(x)) \geq 0
\quad
\bigl(\textup{equivalently, } u(x) \geq \mathfrak{a}(x,u)\bigr),
\quad \textup{for all } x \in \Omega.
\]

We say that~$u \in \mathcal{X}$ is a \emph{classical subsolution} of~\eqref{eq:BVPDPP} if~$
u(x) \leq g(x)$ for all~$x \in \RR^N \setminus \Omega$,
and
\[
\mathcal{A}(x,u,u(x)) \leq 0
\quad
\bigl(\textup{equivalently, } u(x) \leq \mathfrak{a}(x,u)\bigr),
\quad \textup{for all } x \in \Omega.
\]

Finally,~$u \in \mathcal{X}$ is called a \emph{classical solution} if it is both a classical subsolution and a classical supersolution.
\end{Definition}

Note that we assume that we have an explicit~$\mathcal{A}$ but, in general,~$\mathfrak{a}$ is not explicit, the only information we have about~$\mathfrak{a}$ is through its relation with~$\mathcal{A}$.

\begin{Remark}\label{rem:schemewithboundary}
One may regard the exterior datum~$g$ either as part of the operator~$\mathcal{A}$ or as part of the mean value operator~$\mathfrak{a}$. To this end, we define
\[
\mathcal{B}(x,\varphi,s) \coloneqq
\begin{cases}
\mathcal{A}(x,\varphi,s), & x \in \Omega, \\[1mm]
s - g(x), & x \in \RR^N \setminus \Omega,
\end{cases}
\]
or, equivalently,
\[
\mathfrak{b}(x,\varphi) \coloneqq
\begin{cases}
\mathfrak{a}(x,\varphi), & x \in \Omega, \\[1mm]
g(x), & x \in \RR^N \setminus \Omega.
\end{cases}
\]

Then,~$u \in \mathcal{X}$ is a classical supersolution of~\eqref{eq:BVPDPP} if and only if
\[
\mathcal{B}(x,u,u(x)) \geq 0
\quad
\bigl(\textup{equivalently, } u(x) \geq \mathfrak{b}(x,u)\bigr)
\quad \textup{for all } x \in \RR^N,
\]
and analogously for classical subsolutions. Thus, solving the DPP with boundary values~\eqref{eq:BVPDPP} according to~\Cref{def.clasica} is equivalent to solving~$\mathcal{B}(x,u,u(x)) = 0$ 
(or equivalently,~$u(x) = \mathfrak{b}(x,u)$)
for all ~$x \in \RR^N$.

Note that assumptions~\textup{\ref{asA-it:a}--\ref{asA-it:abis}--\ref{asA-it:b}} are also satisfied by~$\mathcal{B}$, and that the operator~$\mathfrak{b}$ remains nondecreasing in its second variable.
\end{Remark}

We now introduce a relaxation of the previous definition based on the ideas of viscosity solutions for PDEs. 
Recall that we work with the spaces~$\mathcal{X}$ and~$\overline{\mathcal{X}}$, with the inclusion
\[
\mathcal{X} \subset \overline{\mathcal{X}}.
\]
\begin{Definition}[Viscosity solutions of the DPP]\label{def:DPPviscosity}
Let~$\Omega \subset \mathbb{R}^N$ be an open domain, let~$g:\RR^N\setminus \Omega \to \RR$, and let~$
\mathcal{A}\colon \Omega \times \mathcal{X} \times \mathbb{R} \to \mathbb{R}$
satisfy assumptions~\textup{\ref{asA-it:a}--\ref{asA-it:abis}--\ref{asA-it:b}}. 
Let~$\mathfrak{a}\colon \Omega \times \mathcal{X} \to \mathbb{R}$ be defined as in~\Cref{rem:defAa}.

We say that~$u \in \overline{\mathcal{X}}$ is a \emph{viscosity supersolution} of~\eqref{eq:BVPDPP} if
$u(x) \geq g(x)$ for all ~$x \in \RR^N \setminus \Omega$,
and for all~$x\in \Omega$ we have that
\[
\mathcal{A}(x,\varphi,u(x))\geq0 \quad (\textup{or equivalently}, \quad 
u(x) \geq \mathfrak{a}(x,\varphi)),
\quad \text{for all } \varphi \in \mathcal{X} \textup{ such that } \varphi\leq u.
\]

We say that~$u \in \overline{\mathcal{X}}$ is a \emph{viscosity subsolution} of~\eqref{eq:BVPDPP} if
$u(x) \leq g(x)$ for all ~$x \in \RR^N \setminus \Omega$,
and for all~$x\in \Omega$ we have that
\[
\mathcal{A}(x,\varphi,u(x))\leq0 \quad (\textup{or equivalently}, \quad 
u(x) \leq \mathfrak{a}(x,\varphi)),
\quad \text{for all } \varphi \in \mathcal{X} \textup{ such that } \varphi\geq u.
\]

Finally, a function~$u \in \mathcal{X}$ is called a \emph{viscosity solution} if it is both a viscosity subsolution and a viscosity supersolution.
\end{Definition}

Although at first glance this may seem different from the usual definition of viscosity solutions for PDEs (see~\Cref{def:viscosityF} below), it is based on the same fundamental idea. Namely, one tests the equation against suitably ``good'' functions for which the operator is well defined, and then characterizes super- and subsolutions via comparison, taking advantage of the monotonicity of the equation.

The key point of the above definition is that, for viscosity subsolutions or supersolutions, only boundedness of~$u$ ($u\in \overline{\mathcal{X}}$) is required, and the quantity~$\mathfrak{a}(x,u)$ may not even be well defined (since~$\mathfrak{a}:\Omega \times \mathcal{X}\to \RR$), for instance due to measurability issues, as discussed above.

There is another possible definition of a viscosity notion of solution to the DPP (although later shown to be equivalent), which allows us to treat the solution in a pointwise sense. Given an operator~$\mathcal{A} \colon \Omega \times \mathcal{X} \times \RR$ (resp.\@~$\mathfrak{a} \colon \Omega \times \mathcal{X} \to \mathbb{R}$), we define the extended operators
\[
\overline{\mathcal{A}}, \underline{\mathcal{A}} \colon \Omega \times \overline{\mathcal{X}} \times \RR \qquad (\text{resp.\@ } \overline{\mathfrak{a}}, \, \underline{\mathfrak{a}} \colon \Omega \times \overline{\mathcal{X}} \to \mathbb{R})
\]
as follows:
\begin{align}
\overline{\mathcal{A}}(x,\psi,s)\coloneqq \inf_{\substack{\varphi \in \mathcal{X} \\ \varphi \leq \psi}} \mathcal{A}(x,\varphi,s), \quad (\text{resp.\@ } \overline{\mathfrak{a}}(x,\psi)
&\coloneqq
\sup_{\substack{\varphi \in \mathcal{X} \\ \varphi \leq \psi}}
\mathfrak{a}(x,\varphi),)
\end{align}
and
\begin{align}
\underline{\mathcal{A}}(x,\psi,s)\coloneqq \sup_{\substack{\varphi \in \mathcal{X} \\ \varphi \geq \psi}} \mathcal{A}(x,\varphi,s), \quad (\text{resp.\@ }  \underline{\mathfrak{a}}(x,\psi)
&\coloneqq
\inf_{\substack{\varphi \in \mathcal{X} \\ \varphi \geq \psi}}
\mathfrak{a}(x,\varphi).)
\end{align}

\begin{Remark}\label{rem:propaext}
It is straightforward to verify that both~$\overline{\mathcal{A}}$ and~$\underline{\mathcal{A}}$ are nonincreasing (resp.\@~$\overline{\mathfrak{a}}$ and~$\underline{\mathfrak{a}}$ are nondecreasing) in the second variable. 
Moreover, one easily checks that
\[
\overline{\mathcal{A}}(x,\psi,s) \leq \underline{\mathcal{A}}(x,\psi,s) \quad (\textup{resp.\@ } \overline{\mathfrak{a}}(x,\psi) \leq \underline{\mathfrak{a}}(x,\psi))
\quad \text{for all } \psi \in \overline{\mathcal{X}}.
\]
Finally, if~$\psi \in \mathcal{X}$, the monotonicity of~$\mathfrak{a}$ in the second variable implies that
\[
\overline{\mathcal{A}}(x,\psi,s) =  \mathcal{A}(x,\psi,s) =\underline{\mathcal{A}}(x,\psi,s) \quad (\textup{resp.\@ } \underline{\mathfrak{a}}(x,\psi)
= \mathfrak{a}(x,\psi)
= \overline{\mathfrak{a}}(x,\psi)).
\]
In particular, this implies that any classical subsolution (respectively, supersolution) is also a viscosity subsolution (respectively, supersolution).
\end{Remark}

\begin{Remark}
The idea of using the monotonicity of an operator to extend its definition to a more general class of functions via a supremum or an infimum is widespread in analysis. For instance, in the definition of the Lebesgue integral, one first defines the integral for simple functions and then extends it to general nonnegative functions~$f$ by setting
\begin{equation}
    \overline{\int} f = \sup\left\{\int \phi : \, \phi \leq f,\ \phi \text{ simple} \right\}.
\end{equation}
Only when~$f$ is measurable does this quantity coincide with the usual Lebesgue integral. Since we will observe that this definition is essentially of viscosity type, it can be regarded as a viscosity approach to integration.
\end{Remark}

\normalcolor

We are now ready to give a second equivalent generalized notion of solution for the DPP with boundary values~\eqref{eq:BVPDPP}.

\begin{Definition}[Viscosity solution of the DPP]\label{def.weak}
Let~$\Omega \subset \mathbb{R}^N$ be an open domain, let~$g:\RR^N\setminus \Omega \to \RR$, and let~$
\mathcal{A}\colon \Omega \times \mathcal{X} \times \mathbb{R} \to \mathbb{R}$
satisfy assumptions~\textup{\ref{asA-it:a}--\ref{asA-it:abis}--\ref{asA-it:b}}. 
Let~$\mathfrak{a}\colon \Omega \times \mathcal{X} \to \mathbb{R}$ be defined as in~\Cref{rem:defAa}, and let
$\underline{\mathfrak{a}}, \overline{\mathfrak{a}}\colon \mathbb{R}^N \times \overline{\mathcal{X}} \to \mathbb{R}$
denote the corresponding extensions.

We say that~$u \in \overline{\mathcal{X}}$ is a \emph{viscosity supersolution} of~\eqref{eq:BVPDPP} if
$u(x) \geq g(x)$ for all ~$x \in \RR^N \setminus \Omega$,
and
\[
\overline{\mathcal{A}}(x,u,u(x))\geq0 \quad (\textup{resp.\@ } u(x) \geq \overline{\mathfrak{a}}(x,u)),
\quad \text{for all } x \in \Omega.
\]

We say that~$u \in \overline{\mathcal{X}}$ is a \emph{viscosity subsolution} of~\eqref{eq:BVPDPP} if
$u(x) \leq g(x)$ for all ~$x \in \RR^N \setminus \Omega$,
and
\[
\underline{\mathcal{A}}(x,u,u(x))\leq 0 \quad (\textup{resp.\@ } u(x) \leq \underline{\mathfrak{a}}(x,u)),
\quad \text{for all } x \in \Omega.
\]

Finally, a function~$u \in \mathcal{X}$ is called a \emph{viscosity solution} if it is both a viscosity subsolution and a viscosity supersolution.
\end{Definition}

We will verify that both definitions of viscosity solutions are equivalent in~\Cref{prop-equiv}.

\subsection{Existence of viscosity solutions of the DPP.}
Our first objective is to establish the existence of viscosity solutions to the DPP with boundary values~\eqref{eq:BVPDPP}. 
The hypotheses of the following result are based on the classical Perron method and on ideas from~\cite{LiuSchikorra}. 
The proof is given in~\Cref{se:existencia}.

\begin{Theorem}\label{teo.1.intro}
Let~$\Omega \subset \mathbb{R}^N$ be an open domain, let~$g :\RR^N\setminus \Omega \to \RR$ bounded, and let~$
\mathcal{A}\colon \Omega \times \mathcal{X} \times \mathbb{R} \to \mathbb{R}$
satisfy assumptions~\textup{\ref{asA-it:a}--\ref{asA-it:abis}--\ref{asA-it:b}}.
Assume that either
\begin{flalign}
    &\text{there exists at least one viscosity subsolution of~\eqref{eq:BVPDPP},} \tag{H$_1$} \label{H1}\\
    &\text{and all viscosity subsolutions of~\eqref{eq:BVPDPP} are uniformly bounded from above,} \tag{H$_2$} \label{H2}
\end{flalign}
or
\begin{flalign}
    &\text{there exists at least one viscosity supersolution of~\eqref{eq:BVPDPP},} \tag{H$_1^*$} \label{H1e}\\
    &\text{and all viscosity supersolutions of~\eqref{eq:BVPDPP} are uniformly bounded from below.} \tag{H$_2^*$} \label{H2e}
\end{flalign}
Then, there exists at least one viscosity solution of~\eqref{eq:BVPDPP}.
\end{Theorem}

It is important to note that the previous result is not concerned with neither uniqueness nor comparison of solutions. We will discuss this in~\Cref{se:comparison}, providing some counterexamples.

\subsection{Comparison results and stability. }
Next, we establish a comparison result between a classical strict supersolution (resp.\@ subsolution) and a viscosity subsolution (resp.\@ supersolution).

\begin{Theorem} \label{thm:striccomparison}
Let~$\Omega \subset \mathbb{R}^N$ be an open domain, let~$g : \mathbb{R}^N \setminus \Omega \to \RR$ be bounded, and let 
$\mathcal{A}\colon \mathbb{R}^N \times \mathcal{X} \times \mathbb{R} \to \mathbb{R}$
satisfy assumptions~\textup{\ref{asA-it:a}--\ref{asA-it:abis}--\ref{asA-it:b}}. Additionally, assume the following:
\begin{enumerate}
\item \textup{\textbf{Stability of~$\mathcal{X}$ by translations}}. We have~$\mathcal{X}+c \subset \mathcal{X}$, in the sense that if~$\varphi \in \mathcal{X}$, then~$\varphi + c \in \mathcal{X}$ for any~$c \in \RR$.
\item \textup{\textbf{Translation invariance}}. Given~$\phi \in \mathcal{X}$ and~$C \in \RR$, we have
\[
\mathcal{A}(x,\phi+C,\phi(x)+C)=\mathcal{A}(x,\phi,\phi(x)).
\]
\item \textup{\textbf{Continuity in the third variable}}. Given~$\phi \in \mathcal{X}$, there exists a modulus of continuity~$\Lambda$ such that
\[
|\mathcal{A}(x,\phi,\phi(x))-\mathcal{A}(x,\phi,\phi(x)-\varepsilon)|\leq \Lambda(\varepsilon)
\quad \textup{as} \quad \varepsilon \to 0^+ \quad \textup{for all} \quad x \in \Omega.
\]
\end{enumerate}
Let~$\overline{u} \in \mathcal{X}$ (resp.\@~$\underline{u} \in \mathcal{X}$) be a classical strict supersolution (resp.\@ subsolution), that is, there exists~$\delta > 0$ such that
\begin{equation}
\begin{cases}
\mathcal{A}(x,\overline{u},\overline{u}(x)) \geq \delta, & x \in \Omega, \\
\overline{u}(x) \geq g(x), & x \in \RR^N \setminus \Omega,
\end{cases}
\quad \left(\textup{resp.\@ } 
\begin{cases}
\mathcal{A}(x,\underline{u},\underline{u}(x)) \leq -\delta, & x \in \Omega, \\
\underline{u}(x) \leq g(x), & x \in \RR^N \setminus \Omega
\end{cases}
\right).
\end{equation}
Then, given any viscosity subsolution (resp.\@ supersolution)~$v \in \overline{\mathcal{X}}$, we have~$v \leq \overline{u}$ in~$\RR^N$ (resp.\@~$v \geq \underline{u}$ in~$\RR^N$).
\end{Theorem}

This result is proved in~\Cref{se:comparison}. As a consequence, if there exists a classical strict supersolution (resp.\@ subsolution), then~\eqref{H2} and~\eqref{H1e} (resp.~\eqref{H2e} and~\eqref{H1}) hold.

\subsection{Convergence of viscosity solutions of the DPP.}
Next, we assume that the DPP is related to an
approximation of a certain elliptic PDE boundary value problem
\begin{equation} \label{eq:BVPPDE} 
\begin{cases}
    F\left(x, u(x), \nabla u(x), D^2u(x) \right) = 0, \quad &x\in \Omega,\\
    u(x)=g(x),  \quad &x\in \partial \Omega.
\end{cases}
\end{equation} 
We will show that the \emph{stability} hypothesis of the seminal paper of Barles and Souganidis~\cite{BarlesSouganidis} can be relaxed so as to require only the existence of viscosity solutions of the approximation scheme (instead of classical ones) in order to ensure convergence of the scheme. Furthermore, if there exists a classical strict supersolution and a classical strict subsolution, both bounded independently of~$\rho$, then we can easily check the stability hypothesis thanks to the result of the previous section.

The following hypotheses on~$F$ are standard when working in the framework of viscosity solutions, although we do not make any explicit use of them in our work.
More precisely, assume~$F$ is \emph{proper}, that is, nondecreasing in the second variable, and that~$F$ is \emph{elliptic}, that is, for all~$(x,s,p)\in \Omega\times\RR\times \RR^N$,
\[
F(x,s,p,\overline{M})\leq F(x,s,p,\underline{M})\quad \textup{for all} \quad \overline{M},\underline{M}\in S^N \quad \textup{such that} \quad \overline{M} \geq \underline{M},
\]
where~$M\geq N$ denotes the usual partial ordering in~$S^N$, the space of~$N\times N$ symmetric matrices.

\begin{Remark}
Under this definition of ellipticity, observe that our model becomes~$F(x,s,p,X) = -\operatorname{trace}(X)$, associated with the PDE~$-\Delta u(x)=0$.
\end{Remark}
Let~$\mathcal{B}_\rho$ be the scheme that incorporates the boundary condition (see~\Cref{rem:schemewithboundary}), that is,
\[
\mathcal{B}_\rho(x,\varphi,s) \coloneqq
\begin{cases}
\mathcal{A}_\rho(x,\varphi,s), & x \in \Omega, \\[1mm]
s - g(x), & x \in \RR^N \setminus \Omega.
\end{cases}
\]

The scheme~$\mathcal{B}_\rho$ is related to the PDE with boundary values~\eqref{eq:BVPPDE} through the following \emph{consistency} assumption: For all~$x\in \RR^N$ and~$\varphi \in C_{b}^\infty(\RR^N)$, there exists a positive modulus of continuity~$\omega$ such that
\begin{align}\label{cons1}
    \limsup_{\rho\to0^+,\
    y\to x,\
    \xi\to0} & \mathcal{B}_\rho(y, \varphi+\xi, \varphi(y)+\xi+\omega(\rho))  \\
    &\leq \begin{cases}
        F^*(x,\varphi(x), \nabla \varphi(x), D^2 \varphi(x)), \quad \textup{if} \quad x\in \Omega,\\
        \max\{\varphi(x)-g_*(x), F^*(x,\varphi(x), \nabla \varphi(x), D^2 \varphi(x))\}\quad \textup{if} \quad x\in \partial\Omega,
    \end{cases}
\end{align}
and
\begin{align}\label{cons2}
    \liminf_{\rho\to0^+,\
    y\to x,\
    \xi\to0}  &\mathcal{B}_\rho(y, \varphi+\xi, \varphi(y)+\xi -\omega(\rho))  \\
    &\geq \begin{cases}
        F_*(x,\varphi(x), \nabla \varphi(x), D^2 \varphi(x)), \quad \textup{if} \quad x\in \Omega,\\
        \min\{\varphi(x)-g^*(x), F_*(x,\varphi(x), \nabla \varphi(x), D^2 \varphi(x))\}\quad \textup{if} \quad x\in \partial\Omega.
    \end{cases}
\end{align}
The precise definitions of the upper and lower envelopes~$F^*,F_*,g^*$ and~$g_*$ will be stated in~\Cref{se:BS}. With a positive modulus of continuity we mean a continuous and strictly increasing function~$\omega:\RR_+\to \RR_+$ such that~$\omega(0)=0$ (so that~$w(\rho)>0$ for all~$\rho>0$).

Finally, we introduce the relaxed \emph{weak stability} hypothesis:
\[
\textup{For all~$\rho>0$, there exists a \emph{viscosity} solution~$u_\rho\in \overline{\mathcal{X}}$ of~\eqref{eq:BVPDPP}, with a bound independent of~$\rho$.}
\]

The following result allows one to construct viscosity subsolutions and supersolutions as half-relaxed limits of solutions to the approximation scheme. The definition of viscosity solutions will be given later in~\Cref{se:BS}, together with the proofs of the results.

\begin{Theorem} \label{teo.convergencia.intro.bis}
Let~$\Omega \subset \mathbb{R}^N$ be an open domain,~$C_b^\infty(\RR^N)\subset \mathcal{X}$,~$g \in \overline{\mathcal{X}}$, and~$\mathcal{A}_\rho\colon \Omega \times \mathcal{X} \times \mathbb{R} \to \mathbb{R}$
satisfy assumptions~\textup{\ref{asA-it:a}--\ref{asA-it:abis}--\ref{asA-it:b}} for each~$\rho>0$. Additionally, assume that the scheme is consistent and weakly stable. Then,
\begin{equation}
        \overline{u}(x) \coloneqq \limsup_{\rho \to 0, \ y \to x} u_\rho(y)
        \quad \text{and} \quad
        \underline{u}(x) \coloneqq \liminf_{\rho \to 0, \ y \to x} u_\rho(y)
\end{equation}
are, respectively, an usc.\@ viscosity subsolution and a lsc.\@ viscosity supersolution of problem~\eqref{eq:BVPPDE}.
\end{Theorem}

\begin{Remark}\label{rem:restrictedTest}
    Note that, in some cases, the definition of viscosity solution involves test functions that are more restrictive than~$\cont_b^\infty(\RR^N)$. 
For instance, in the case of~$p$-harmonic functions, one can additionally assume that~$\nabla \varphi(x)\neq 0$ (see~\cite{JuutinenLindqvistManfredi2001, AsymptoticMeanValuePharmonic}). 
In general, given a class of test functions~$\mathcal{T}$ that defines a notion of viscosity solution, the previous equivalence still holds as long as~$\mathcal{T} \subset \mathcal{X}$. 
The same observation applies in nonlocal problems, where the concept of viscosity solution may vary  from the standard local one.
\end{Remark}

Finally, if the limit problem satisfies a special type of comparison principle (introduced in~\cite{BarlesSouganidis} as the \emph{strong uniqueness} property), then the scheme converges. 

\begin{Corollary}\label{cor:convSchm}
Let the assumptions of~\Cref{teo.convergencia.intro.bis} hold. Additionally, assume the following \emph{strong uniqueness} property:
If~$u$ is an usc.\@ viscosity subsolution and~$v$ is a lsc.\@ viscosity supersolution of problem~\eqref{eq:BVPPDE}, then~$u\leq v$ in~$\overline{\Omega}$.

Then, the sequence of viscosity solutions~$\{u_\rho\}_{\rho>0}$ of~\eqref{eq:BVPDPP} converges locally uniformly to the unique viscosity solution of~\eqref{eq:BVPPDE}.
\end{Corollary}

The previous results (and their proofs) are similar to the ones in~\cite{BarlesSouganidis}.  
This is precisely what we were aiming for: even though we are using an extended notion of solution for the approximating problem, the convergence result remains valid.  
This is a consequence of the fact that~$\cont^\infty_b(\RR^N) \subset \mathcal{X}$, which implies that results that only depend on test functions~$\varphi \in \cont^\infty_b(\RR^N)$, such as the previous one, remain true when considering viscosity solutions (since~$\mathfrak{a}(x,\varphi)=\underline{\mathfrak{a}}(x,\varphi)=\overline{\mathfrak{a}}(x,\varphi)$ for all~$\varphi \in \mathcal{X}$).

\subsection{Asymptotic mean value formulas in the viscosity sense.} The final goal of this paper is to prove that there exists an asymptotic mean value formula given by the DPP such that it characterizes viscosity solutions to the PDE. In order to do so, first we give a precise meaning to being a viscosity super- and subsolution of the asymptotic mean value formula
\begin{equation} \label{eq:asymptoticmeanvalueformula}
    \mathcal{A}_\rho(x,u,u(x)+ o(\rho))=0, 
    \quad \left( \textup{or} \quad u(x)=\mathfrak{a}_\rho(x,u)-o(\rho)\right),
    \quad \textup{as } \rho\to0^+.
\end{equation}
To show that~$u$ is a viscosity solution of this problem if and only if it is a viscosity solution of the PDE, we require a consistency property linking the two notions. In this setting, for any~$x\in \Omega$, any~$\phi \in \cont_b^\infty(\RR^N)$ and some~$\Lambda \colon \RR_+ \times \Omega \to \RR$ such that~$\lim_{\rho \to 0} \Lambda(\rho, x) = 0$, the consistency assumption can be divided in the following two:
\begin{equation} \label{eq:DPPtoPDE}
    \begin{aligned}
          &\liminf_{\rho \to 0} \mathcal{A}_\rho(x,\phi, \phi(x) + \Lambda(\rho,x) ) \leq F^*(x, \phi(x), \nabla\phi(x), D^2\phi(x)), \\
          \big(\textup{resp.} \quad  &\limsup_{\rho \to 0} \mathcal{A}_\rho(x,\phi, \phi(x) + \Lambda(\rho,x) ) \geq F_*(x, \phi(x), \nabla\phi(x), D^2\phi(x))\big),
    \end{aligned}
\end{equation}
and
\begin{equation} \label{eq:PDEtoDPP}
    \begin{aligned}
        &\liminf_{\rho \to 0} \mathcal{A}_\rho(x,\phi, \phi(x) + \Lambda(\rho,x) ) \geq F^*(x, \phi(x), \nabla\phi(x), D^2\phi(x)). \\
        \big(\textup{resp.} \quad  &\limsup_{\rho \to 0} \mathcal{A}_\rho(x,\phi, \phi(x) + \Lambda(\rho,x) ) \leq F_*(x, \phi(x), \nabla\phi(x), D^2\phi(x))\big).
    \end{aligned}
\end{equation}         
Under these conditions we have the result below.
\begin{Theorem} \label{teo.asympt.intro}
If~$u$ is a viscosity solution of~\eqref{eq:asymptoticmeanvalueformula} and the consistency condition~\eqref{eq:DPPtoPDE} holds, then~$u$ is a viscosity solution of the PDE~\eqref{eq:BVPPDE}. Reciprocally, if~$u$ is a viscosity solution of the PDE~\eqref{eq:BVPPDE} and the consistency condition~\eqref{eq:PDEtoDPP} holds, then~$u$ is a viscosity solution of~\eqref{eq:asymptoticmeanvalueformula}.
\end{Theorem}
A more accurate statement of this result, together with the precise notion of viscosity solution of~\eqref{eq:asymptoticmeanvalueformula}, will be given in~\Cref{se:asymptotic}.

\section{Existence of viscosity solutions to the DPP}
\label{se:existencia}

First, we prove that the two definitions of viscosity solution of the DPP are equivalent.

\begin{Proposition}\label{prop-equiv}
The notions of viscosity supersolutions (resp.\@ subsolutions) given in Definitions~\ref{def.clasica} and~\ref{def:DPPviscosity} are equivalent.
\end{Proposition}

\begin{proof}
  We prove the result for viscosity supersolutions. 
The boundary data are the same in both definitions, so we focus on what happens inside~$\Omega$. 
If~$u$ is a viscosity supersolution in the sense of~\Cref{def.clasica}, then for every~$x \in \Omega$ and every~$\varphi\in \mathcal{X}$ such that~$\varphi \leq u$ we have that
\begin{equation}
u(x) \geq \overline{\mathfrak{a}}(x,u)= \sup_{\substack{\psi \in \mathcal{X} \\ \psi \leq u}}
\mathfrak{a}(x,\psi) \geq \mathfrak{a}(x,\varphi).
\end{equation}
Hence,~$u$ is a viscosity supersolution in the sense of~\Cref{def:DPPviscosity}.
Conversely, assume that for all~$x\in \Omega$ we have that 
\begin{equation}
u(x) \geq \mathfrak{a}(x,\varphi),
\end{equation}
for every~$\varphi \in \mathcal{X}$ with~$\varphi \leq u$. Taking the supremum over all such~$\varphi$ yields~$u(x) \geq \overline{\mathfrak{a}}(x,u)$,
which shows that~$u$ is a viscosity supersolution in the sense of~\Cref{def.clasica}. 

Finally, let us verify the equivalence of the notions of viscosity solutions in the formulation given by~$\mathcal{A}$. 
First, observe that for any~$\varphi \in \mathcal{X}$ such that~$\varphi \leq u$, the following statements are equivalent:
\[
u(x) \geq \mathfrak{a}(x,\varphi),
\]
and
\[
\mathcal{A}(x,\varphi,u(x)) \geq 0.
\]
This follows from the fact that ~$\mathcal{A}$ is nondecreasing in its last variable and~$\mathcal{A}(x,\varphi,\mathfrak{a}(x,\varphi))=0$ by definition. 
Moreover, the latter condition is also equivalent to
\[
\overline{\mathcal{A}}(x,u,u(x)) \geq 0,
\]
by the very definition of~$\overline{\mathcal{A}}$ as an infimum over admissible test functions (in the second variable).
\end{proof}

 We now establish the existence of solutions to the DPP stated in~\Cref{teo.1.intro} by means of Perron's method. The solution is given as the supremum of viscosity subsolutions. For the reader's convenience, we recall the assumptions under consideration:
\begin{flalign}
    &\text{there exists at least one viscosity subsolution of~\eqref{eq:BVPDPP},} \tag{H$_1$} \\
    &\text{all viscosity subsolutions of~\eqref{eq:BVPDPP} are uniformly bounded from above.} \tag{H$_2$}
\end{flalign}
The proof under assumptions~\eqref{H1e} and~\eqref{H2e} follows analogously.

\begin{Theorem}[Extended Perron's method]\label{thm:perron}
Let~$\Omega \subset \mathbb{R}^N$ be an open domain, let~$g :\mathbb{R}^N \setminus \Omega \to \RR$ bounded, and let~$
\mathcal{A}\colon \mathbb{R}^N \times \mathcal{X} \times \mathbb{R} \to \mathbb{R}$
satisfy assumptions~\textup{\ref{asA-it:a}--\ref{asA-it:abis}--\ref{asA-it:b}}. 
Assume~\eqref{H1} and~\eqref{H2} (resp.\@~\eqref{H1e} and~\eqref{H2e}).  
Let~$\underline{S}$ (resp.\@~$\overline{S}$) be the set of all viscosity subsolutions (resp.\@ supersolutions) of~\eqref{eq:BVPDPP}. Then,
\begin{equation} \label{capital}
\overline{u} (x) = \begin{cases}
   \displaystyle \sup_{v \in \underline{S}} v(x), \qquad \text{ for } x \in \Omega, \\
    g(x), \qquad \text{ for } x \in \RR^N \setminus \Omega
\end{cases} \quad \left(\text{resp.}  \quad 
\underline{u} (x) = \begin{cases}
   \displaystyle \inf_{v \in \overline{S}} v(x), \qquad \text{ for } x \in \Omega, \\
    g(x), \qquad \text{ for } x \in \RR^N \setminus \Omega
\end{cases}
\right)
\end{equation}
is a viscosity solution of~\eqref{eq:BVPDPP}.
\end{Theorem}

\begin{Remark}
 For a result like the one above to hold in the classical framework, the supremum must belong to the class of functions on which~$\mathcal{A}$ is well defined, namely~$\overline{u} \in \mathcal{X}$. 
This requirement is not satisfied in general. 
In~\cite{LiuSchikorra}, the authors address this difficulty by assuming that~$\mathcal{X}$ is closed under pointwise suprema. 
In contrast, this assumption is not needed here, since~$\overline{u} \in \overline{\mathcal{X}}$ suffices for our definition.
\end{Remark}

\begin{proof} [Proof of~\Cref{thm:perron}]
   We first show that~$\overline{u}$, which is well defined thanks to hypotheses~\eqref{H1} and~\eqref{H2}, is a viscosity subsolution of~\eqref{eq:BVPDPP}. Let~$u \in \underline{\mathcal S}$, which is nonempty by assumption~\eqref{H1}. This implies that~$\overline{u}$ is bounded from below by~$u$ and bounded from above due to assumption~\eqref{H2}. Thus,~$\overline{u}\in \overline{\mathcal{X}}$. Now, 
for every~$x\in\Omega$, since~$u$ is a viscosity subsolution, we have
\[
u(x) \le \underline{\mathfrak{a}}(x,u).
\]
Since~$u\le \overline{u}$ and~$\underline{\mathfrak{a}}$ is nondecreasing in its second argument (see~\Cref{rem:propaext}), it follows that
\[
u(x) \le \underline{\mathfrak{a}}(x,\overline{u}).
\]
Taking the pointwise supremum over all~$u\in\underline{\mathcal S}$ yields
\[
\overline{u}(x) \le \underline{\mathfrak{a}}(x,\overline{u}), \qquad x\in\Omega,
\]
and therefore~$\overline{u}$ is a viscosity subsolution of~\eqref{eq:BVPDPP}.

We now prove that~$\overline{u}$ is also a viscosity supersolution of~\eqref{eq:BVPDPP}.  
Define the function~$v:\RR^N\to\RR$ given by
\[
v(x)=\underline{\mathfrak{a}}(x,\overline{u}), \qquad x\in\Omega \quad \textup{and} \quad v(x)= g(x), \qquad x\in\RR^N\setminus\Omega.
\]
By construction,~$v\in\overline{\mathcal X}$. Moreover, from the previous inequality we already know that
\[
\overline{u}(x) \le v(x).
\]
Using again the monotonicity of~$\underline{\mathfrak{a}}$, we obtain
\[
v(x)=\underline{\mathfrak{a}}(x,\overline{u}) \le \underline{\mathfrak{a}}(x,v),
\]
which shows that~$v$ is a viscosity subsolution of~\eqref{eq:BVPDPP}, that is,~$v\in\underline{\mathcal S}$.  
By definition of~$\overline{u}$ as the pointwise supremum in~$\underline{\mathcal S}$, this implies~$v\le \overline{u}$. Hence,~$v=\overline{u}$.
Finally, using~\Cref{rem:propaext}, we conclude that
\[
\overline{u}(x)=v(x)=\underline{\mathfrak{a}}(x,\overline{u}) \ge \overline{\mathfrak{a}}(x,\overline{u}),
\]
so that~$\overline{u}$ is a viscosity supersolution.  
Together with the boundary condition~$\overline{u}=g$ in~$\mathbb{R}^N\setminus\Omega$, this completes the proof.

\end{proof}
\begin{Remark} 
Notice that for an exterior datum~$g$ that is not measurable, we cannot even consider classical solutions. However, the measurability of~$g$ does not affect at all the previous result of existence for viscosity solutions 
\end{Remark}

\section{Comparison result and weak stability for the DPP} \label{se:comparison}

The aim of this section is to provide tools to establish a comparison result, and hence stability properties, for viscosity solutions of the DPP.

In general, without additional assumptions on the DPP one cannot expect a comparison result to hold.
Nevertheless, under very mild and natural assumptions on~$\mathcal{A}$, it is still possible to compare viscosity subsolutions (resp.\@ supersolutions) with classical strict supersolutions (resp.\@ subsolutions).

Recall that we assume the following extra conditions:

\begin{enumerate}
\item\label{visesta-item1} \textup{\textbf{Stability of~$\mathcal{X}$ by translations}}. We have~$\mathcal{X}+c \subset \mathcal{X}$, in the sense that if~$\varphi \in \mathcal{X}$, then~$\varphi + c \in \mathcal{X}$ for any~$c \in \RR$.
\item\label{visesta-item2} \textup{\textbf{Translation invariance}}. Given~$\phi \in \mathcal{X}$ and~$C \in \RR$, we have
\[
\mathcal{A}(x,\phi+C,\phi(x)+C)=\mathcal{A}(x,\phi,\phi(x)).
\]
\item\label{visesta-item3} \textup{\textbf{Continuity in the third variable}}. Given~$\phi \in \mathcal{X}$, there exists a modulus of continuity~$\Lambda$ such that
\[
|\mathcal{A}(x,\phi,\phi(x))-\mathcal{A}(x,\phi,\phi(x)-\varepsilon)|\leq \Lambda(\varepsilon)
\quad \textup{as} \quad \varepsilon \to 0^+ \quad \textup{for all} \quad x \in \Omega.
\]
\end{enumerate}
\normalcolor

\begin{proof}[Proof of~\Cref{thm:striccomparison}]
We do it only for the classical strict supersolution case. Assume, by contradiction, that 
\[
M \coloneqq \sup_{\RR^N}(v-\overline{u})>0.
\]
Due to the ordering in~$\RR^N \setminus \Omega$, we have
\[
M = \sup_{\Omega}(v-\overline{u}).
\]
By definition of the supremum, there exists a sequence~$\{x_\varepsilon\}_{\varepsilon>0} \subset \Omega$ such that for all~$\varepsilon>0$ we have
\[
M-\varepsilon \leq v(x_\varepsilon)-\overline{u}(x_\varepsilon)\leq M.
\]
Finally, note that~$v \leq \overline{u}+M \in \mathcal{X}$ by~\eqref{visesta-item1}. Thus,~$\overline{u}+M$ can be used as a test function in the definition of viscosity subsolution of the DPP. Then,
\begin{align}
0 &\geq \mathcal{A}(x_\varepsilon,\overline{u}+M,v(x_\varepsilon)) \\
&\geq \mathcal{A}(x_\varepsilon,\overline{u}+M,\overline{u}(x_\varepsilon)+M-\varepsilon) \\
&\stackrel{\eqref{visesta-item2}}{=} \mathcal{A}(x_\varepsilon,\overline{u},\overline{u}(x_\varepsilon)-\varepsilon) \\
&\stackrel{\eqref{visesta-item3}}{\geq} \mathcal{A}(x_\varepsilon,\overline{u},\overline{u}(x_\varepsilon)) - \Lambda(\varepsilon) \\
&\geq \delta - \Lambda(\varepsilon).
\end{align}
This is a contradiction since we can take~$\varepsilon>0$ small enough so that~$\delta - \Lambda(\varepsilon)>0$.
\end{proof}

\begin{Remark}
Note that the above result ensures~\eqref{H2} and~\eqref{H2e} as long as one can construct a classical strict supersolution and a classical strict subsolution respectively. Additionally, since classical sub- and supersolutions are viscosity sub- and supersolutions, ~\eqref{H1} and~\eqref{H1e} are also satisfied.

A possible simple strategy is the following: find a smooth function~$\phi$ satisfying
\begin{equation}
\begin{cases}
F\bigl(x,\phi(x),\nabla \phi(x),D^2\phi(x)\bigr)\geq1, & x \in \Omega,\\
\phi(x)\geq \displaystyle\sup_{\RR^N\setminus\Omega} g, & x \in \RR^N\setminus\Omega.
\end{cases}
\end{equation}
 Then, by the following consistency assumption (which is weaker than the ones typically needed for other results in the paper),
\[
\sup_{x\in \Omega}\left|F\bigl(x,\phi(x),\nabla \phi(x),D^2\phi(x)\bigr)- \mathcal{A}_\rho(x,\phi,\phi(x))\right|
\leq \|\phi\|_{C^{k}(\RR^N)}\, \overline{\Lambda}(\rho),
\]
for some modulus of continuity~$\overline{\Lambda}$, we can choose~$\rho>0$ small enough so that~$\phi$ satisfies
\[
\mathcal{A}_\rho(x,\phi,\phi(x)) \geq \frac{1}{2} \quad \textup{in } \Omega.
\]
\end{Remark}

Even under the previous conditions~\eqref{visesta-item1}–\eqref{visesta-item2}–\eqref{visesta-item3}, the comparison principle or even a uniqueness result will not hold, as we show in the following examples.

\begin{Example}[No uniqueness nor comparison]
    Let us define~$\mathfrak{a}:B_1(0) \times \overline{\mathcal{X}} \to \RR$ as 
    \begin{equation}
        \mathfrak{a}(x,\varphi) = \varphi(0),
    \end{equation}
    and also~$\mathcal{A}: B_1(0) \times \overline{\mathcal{X}} \times \RR \to \RR$ in the standard way, 
    \begin{equation}
        \mathcal{A}(x, \varphi, s) = s -\mathfrak{a}(x,\varphi) \eqqcolon s - \varphi(0).
    \end{equation}
    Choose the exterior datum~$g \equiv 0$ in~$\RR^N \setminus B_1(0)$. In this example, it is natural that~$\mathcal{X} = \overline{\mathcal{X}}$ since there is no measurability involved. Observe that the defined~$\mathcal{A}$ satisfies~\textup{\ref{asA-it:a}--\ref{asA-it:abis}--\ref{asA-it:b}}, ~\eqref{visesta-item1}–\eqref{visesta-item2}–\eqref{visesta-item3} and is even linear. However, there is no comparison principle nor uniqueness for the corresponding DPP equation~$\mathcal{A}(x, u , u(x) )= 0$, even though we are using classical solutions, since any function defined as 
    \begin{equation}
        u(x) = \begin{cases}
                   \displaystyle C, \qquad \text{ for } x \in B_1(0), \\
                    0, \qquad \text{ for } x \in \RR^N \setminus B_1(0)
                \end{cases}
    \end{equation}
    is a solution for any~$C \in \RR$ (and these are the only solutions). 
\end{Example}

\begin{Example}[Uniqueness but no comparison]
    If in the prior example we consider~$\mathcal{X}$ as the set of continuous functions on~$\RR^N$, notice that the only solution is~$u \equiv 0$. However, even in this restricted setting there is no comparison principle between continuous super- and subsolution. Define
    \begin{equation}
        u(x) = \begin{cases}
           \displaystyle 1-|x|^2, \qquad \text{ for } x \in B_1(0), \\
            0, \qquad \text{ for } x \in \RR^N \setminus B_1(0),
        \end{cases} \qquad
        v(x) = \begin{cases}
           \displaystyle |x|^2-1, \qquad \text{ for } x \in B_1(0), \\
            0, \qquad \text{ for } x \in \RR^N \setminus B_1(0).
        \end{cases}
    \end{equation}
    Then,~$u$ is a continuous subsolution and~$v$ a continuous supersolution, but~$u \geq v$ in the whole~$\RR^N$. 
\end{Example}

This example is degenerate, in the sense that the DPP operator~$\mathcal{A}(x, \varphi, s)$ does not see the exterior datum~$g$ for any point~$x \in \Omega$, and thus cannot translate the ordering in~$\RR^N \setminus B_1(0)$ to the interior properly, not even when restricting ourselves to the set of continuous functions. 

\section{Convergence of the solutions to the DPP towards the 
solution to a PDE} \label{se:BS}

We consider the boundary value problem
\begin{equation} \label{eq:BVPPDE-sectionproof} 
\begin{cases}
    F\left(x, u(x), \nabla u(x), D^2u(x) \right) = 0, \quad &x\in \Omega,\\
    u(x)=g(x),  \quad &x\in \partial \Omega.
\end{cases}
\end{equation} 

Let us briefly recall the notions of upper semi-continuous and lower semi-continuous envelopes of a function~$z:C\to \RR$ where~$C$ is any subset of~$\RR^d$. They are functions~$z^*, z_* \colon C \to \RR$ defined as
\[
z^*(x)=\limsup_{y\to x\, , \, y \in C} z(y) \quad \textup{and} \quad z_*(x)=\liminf_{y\to x\, , \, y \in C} z(y).
\]

We can define now the concept of viscosity solution.
\begin{Definition}\label{def:viscosityF}
Let~$\Omega$ be an open domain, let~$F:\Omega\times \RR\times \RR^N\times S^N$ be elliptic and proper, and let~$g\in \overline{\mathcal{X}}$.  
A function~$u\in\overline{\mathcal{X}}$ is a viscosity subsolution (resp.\@ supersolution) of~\eqref{eq:BVPPDE-sectionproof} if for all~$\varphi\in C^\infty_b(\RR^N)$ and all~$x\in \overline{\Omega}$ such that~$u^*-\varphi$ (resp.\@~$u_*-\varphi$) has a global maximum (resp.\@ minimum) at~$x$, we have
\begin{align}
    &F_*(x,u^*(x),\nabla \varphi(x),D^2\varphi(x))\leq 0  \quad \textup{if} \quad x\in \Omega,\\
    (\textup{resp.\@ } &F^*(x,u_*(x),\nabla \varphi(x),D^2\varphi(x))\geq 0),
\end{align}
and
\begin{align}
    &\min\{u^*(x) - g^*(x),\,F_*(x,u^*(x),\nabla \varphi(x),D^2\varphi(x)) \}\leq 0 \quad \textup{if} \quad x\in \partial\Omega,\\
    (\textup{resp.\@ } &\max\{u_*(x) - g_*(x),\,F^*(x,u_*(x),\nabla \varphi(x),D^2\varphi(x))\}\geq 0).
\end{align} 
\end{Definition}

We are now ready to prove the convergence results.
\begin{proof}[Proof of~\Cref{teo.convergencia.intro.bis}]
We consider the case of~$\overline{u}$. Since~$\overline{u}$ is upper semicontinuous, we have~$\overline{u}=\overline{u}^*$. 
Let~$x_0$ be a maximum point of~$\overline{u}-\phi$ on~$\overline{\Omega}$ for some~$\phi \in C_b^\infty(\RR^N)$. 
We assume, without loss of generality, that~$x_0$ is a strict global maximum and that~$\overline{u}(x_0)=\phi(x_0)$. Furthermore, fix~$r>0$ and choose~$\phi$ so that~$\phi(x)\geq 2\sup_\rho \|u_\rho\|_\infty$ for all~$x\notin B_r(x_0)$. 
Then there exist sequences~$\{\rho_n\}_{n\in\NN}$ and~$\{y_n\}_{n\in\NN}$ such that
\[
(\rho_n,y_n)\to (0^+,x_0) 
\quad \textup{and} \quad 
u_{\rho_n}(y_n)\to \overline{u}(x_0)
\quad \textup{as } n\to+\infty.
\]
Moreover, passing to a subsequence if necessary, we may assume that for every~$n$ the point~$y_n$ is chosen sufficiently close to the global supremum of~$u_{\rho_n}-\phi$ so that, for any given modulus of continuity~$\omega$, we have
\begin{equation}
\xi_n \coloneqq u_{\rho_n}(y_n)-\phi(y_n)+\omega(\rho_n)
>
\sup_{\overline{B}_r(x_0)} \{u_{\rho_n}-\phi\}.
\end{equation}
It follows that~$\lim_{n\to\infty} \xi_n = 0$ and
\begin{equation}
u_{\rho_n}(x) \leq \phi(x)+\xi_n 
\quad \text{for all } x\in \overline{\Omega}.
\end{equation}

At this point, a slight difference with the original proof in~\cite{BarlesSouganidis} appears. 
Since~$u_{\rho_n}$ is a viscosity solution of the DPP (and in particular a viscosity subsolution), for every~$y_n\in \Omega$ we have
\begin{equation}
u_{\rho_n}(y_n) 
\leq \underline{\mathfrak{a}}_{\rho_n}(y_n,u_{\rho_n})
\leq \underline{\mathfrak{a}}_{\rho_n}(y_n,\phi+\xi_n)
= \mathfrak{a}_{\rho_n}(y_n,\phi+\xi_n),
\end{equation}
where we used that~$\underline{\mathfrak{a}}_{\rho_n}$ is nondecreasing in its second variable and that~$\phi+\xi_n\in C_b^\infty(\RR^N)\subset \mathcal{X}$, so that~$
\underline{\mathfrak{a}}_{\rho_n}(y_n,\phi+\xi_n)
=\overline{\mathfrak{a}}_{\rho_n}(y_n,\phi+\xi_n)
=\mathfrak{a}_{\rho_n}(y_n,\phi+\xi_n)$.
Thus,
\begin{equation}
\phi(y_n)+\xi_n-\omega(\rho_n)
\leq \mathfrak{a}_{\rho_n}(y_n,\phi+\xi_n).
\end{equation}
Thus, 
if~$y_n\in \Omega$, this implies
\begin{align}
\mathcal{B}_{\rho_n}(y_n,\phi+\xi_n,\phi(y_n)+\xi_n-\omega(\rho_n))
&=\mathcal{A}_{\rho_n}(y_n,\phi+\xi_n,\phi(y_n)+\xi_n-\omega(\rho_n)) \\
&\leq \mathcal{A}_{\rho_n}(y_n,\phi+\xi_n,
\mathfrak{a}_{\rho_n}(y_n,\phi+\xi_n))=0.
\end{align}
On the other hand, if~$y_n\in \RR^N\setminus\Omega$, then
\begin{align}
\mathcal{B}_{\rho_n}(y_n,\phi+\xi_n,\phi(y_n)+\xi_n-\omega(\rho_n))
&=\phi(y_n)+\xi_n-\omega(\rho_n)-g(y_n)\\
&=u_{\rho_n}(y_n)-g(y_n)=0.
\end{align}

Finally, using the consistency assumption, we obtain
\begin{align}
0
&\geq \liminf_{n\to+\infty}
\mathcal{B}_{\rho_n}(y_n,\phi+\xi_n,\phi(y_n)+\xi_n-\omega(\rho_n))\\
&\geq \liminf_{\rho\to0^+,\,y\to x_0,\,\xi\to0}
\mathcal{B}_\rho(y,\phi+\xi,\phi(y)+\xi-\omega(\rho))\\
&\geq
\begin{cases}
F_*(x_0,\phi(x_0),\nabla\phi(x_0),D^2\phi(x_0)),
& \textup{if } x_0\in \Omega,\\[1mm]
\min\{\phi(x_0)-g^*(x_0),
F_*(x_0,\phi(x_0),\nabla\phi(x_0),D^2\phi(x_0))\},
& \textup{if } x_0\in \partial\Omega.
\end{cases}
\end{align}
Since~$\overline{u}^*(x_0)=\overline{u}(x_0)=\phi(x_0)$, this shows that~$\overline{u}$ is an upper semicontinuous viscosity subsolution.
\end{proof}

\begin{proof}[Proof of~\Cref{cor:convSchm}]
By construction, we have~$\underline{u}\leq \overline{u}$. 
By~\Cref{teo.convergencia.intro.bis},~$\overline{u}$ is an upper semicontinuous viscosity subsolution and~$\underline{u}$ is a lower semicontinuous viscosity supersolution of problem~\eqref{eq:BVPPDE}. 
Thus, by the strong uniqueness property, we obtain~$\underline{u}\geq \overline{u}$, and therefore
\[
u \coloneqq\underline{u}=\overline{u},
\]
where~$u$ is the unique viscosity solution of~\eqref{eq:BVPPDE}. Pointwise convergence follows since
\[
\limsup_{\rho\to0^+} u_{\rho}(x)
\leq \limsup_{\substack{\rho\to0^+\\ y\to x}} u_{\rho}(y)
= u(x)
= \liminf_{\substack{\rho\to0^+\\ y\to x}} u_{\rho}(y)
\leq \liminf_{\rho\to0^+} u_{\rho}(x).
\]
Finally, locally uniform convergence follows by standard arguments; see, for instance, Lemma~4.1 in~\cite{BarlesBook}.
\end{proof}

\section{Asymptotic mean value formulas in the viscosity sense} \label{se:asymptotic}

Recall that the goal of this section is to understand the equivalence between being a viscosity solution of the PDE
\begin{equation}
F\bigl(x,u(x),\nabla u(x),D^2u(x)\bigr)=0,
\end{equation}
and being a viscosity solution of the asymptotic mean value property
\begin{equation} \label{as:exp}
    \mathcal{A}_\rho(x,u,u(x)+ o(\rho))=0, 
    \quad \left( \textup{or} \quad u(x)=\mathfrak{a}_\rho(x,u)-o(\rho)\right),
    \quad \textup{as } \rho\to 0^+.
\end{equation}
However, we consider that the above formal definition of solution to the asymptotic mean value formula allows for several possible interpretations.

We start the section by defining below what we consider the most natural concept of viscosity solution to the asymptotic mean value formula~\eqref{as:exp}.

\begin{Definition}\label{def:weakviscasexo}
    Let~$\Omega \subset \mathbb{R}^N$ be an open domain,~$C_b^\infty(\RR^N)\subset \mathcal{X}$, let~$\mathcal{A}_\rho\colon \Omega \times \mathcal{X} \times \mathbb{R} \to \mathbb{R}$ satisfy assumptions~\textup{\ref{asA-it:a}--\ref{asA-it:abis}--\ref{asA-it:b}} for each~$\rho>0$, and let~$\omega \coloneq \RR_+ \times \Omega \to \RR_+$ be a function such that~$\lim_{\rho \to 0}\omega(\rho,x)=0$ for all~$x\in \Omega$.
    
 A function~$u\in \overline{\mathcal{X}}$ is a \emph{viscosity subsolution} associated to~$\omega$ of~\eqref{as:exp}  if, for any function~$\Lambda: \RR_+ \times \Omega \to \RR$ satisfying~$|\Lambda(\rho,x)| \leq \omega(\rho,x)$, and for any~$\varphi \in C_b^2(\RR^N)$ satisfying~$u^* \leq \varphi$,~$u^*(x_0) = \varphi(x_0)$, we have that
   \begin{equation} \label{eq:weakvisc1}
        \limsup_{\rho\to0}\mathcal{A_\rho}(x_0,\varphi, \varphi(x_0)+\Lambda(\rho, x_0))\leq 0.
    \end{equation}

   A function~$u\in \overline{\mathcal{X}}$ is a \emph{viscosity supersolution} associated to~$\omega$ of~\eqref{as:exp}  if, for any function~$\Lambda: \RR_+ \times \Omega \to \RR$ satisfying~$|\Lambda(\rho,x)| \leq \omega(\rho,x)$, and for any~$\varphi \in C_b^2(\RR^N)$ satisfying~$u_* \geq \varphi$,~$u_*(x_0) = \varphi(x_0)$, we have that
    \begin{equation} \label{eq:weakvisc2}
        \liminf_{\rho\to0}\mathcal{A_\rho}(x_0,\varphi, \varphi(x_0)+\Lambda(\rho, x_0))\geq 0.
    \end{equation}

    A function~$u\in \overline{\mathcal{X}}$ is a \emph{viscosity solution} associated to~$\omega$ of~\eqref{as:exp} in~$x_0 \in \Omega$  if it is both a  viscosity supersolution and a  viscosity subsolution in~$x_0 \in \Omega$ associated to~$\omega$. 
\end{Definition}

\begin{Remark}
    In the previous definition, one may consider the appearance of~$\Lambda$ in the last variable of the limits~\eqref{eq:weakvisc1} and~\eqref{eq:weakvisc2} unnecessary, and thus the hypothesis involving the existence of~$\omega$ not required. This concern is justified when restricting ourselves to the case where~$\mathfrak{a}$ is given explicitly. This alternative definition can be recovered considering~$\omega=0$, however we keep~$\omega$ to allow for more generality, specifically for the implicit cases.
\end{Remark}
 
\normalcolor

Our following result shows that, under certain consistency conditions,~$u$ a viscosity solution to~\eqref{as:exp} if and only if it is a viscosity solution to 
\begin{equation}\label{eq:defFsecas}
     F\left(x, u(x), \nabla u(x), D^2u(x) \right) = 0, \quad x\in \Omega.
\end{equation}
\begin{Theorem}[Asymptotic expansion~$\Longleftrightarrow$  PDE]\label{thm:weakvisctivisc}  Let~$\Omega \subset \mathbb{R}^N$ be an open domain,~$C_b^\infty(\RR^N)\subset \mathcal{X}$, and~$\mathcal{A}_\rho\colon \mathbb{R}^N \times \mathcal{X} \times \mathbb{R} \to \mathbb{R}$
satisfy assumptions~\textup{\ref{asA-it:a}--\ref{asA-it:abis}--\ref{asA-it:b}} for each~$\rho>0$. The following holds:
\begin{enumerate}[label={\noindent \rm (\alph*)}]
    \item\label{thm:weakvisctivisc-item1} Assume the following consistency condition: For any~$\phi \in \cont^\infty_b(\RR^N)$ and any~$x\in \Omega$ we have that
\begin{equation}
\begin{aligned}
      &\liminf_{\rho \to 0} \mathcal{A}_\rho(x,\phi, \phi(x) + \Lambda(\rho,x) ) \leq F^*(x, \phi(x), \nabla\phi(x), D^2\phi(x)), \\
      \big(\textup{resp.} \quad  &\limsup_{\rho \to 0} \mathcal{A}_\rho(x,\phi, \phi(x) + \Lambda(\rho,x) ) \geq F_*(x, \phi(x), \nabla\phi(x), D^2\phi(x))\big),
      \end{aligned}
\end{equation}         
    for some function~$\Lambda: \RR_+ \times \Omega \to \RR$. Then, if~$u$ is a viscosity supersolution associated to~$\omega$ (resp.\@ subsolution)   as in~\Cref{def:weakviscasexo}, and~$|\Lambda| \leq \omega$, it is also a viscosity supersolution (resp.\@ subsolution) of~\eqref{eq:defFsecas}.
     \item\label{thm:weakvisctivisc-item2}  Assume the following consistency condition: For any~$\phi \in \cont^\infty_b(\RR^N)$ and any~$x\in \Omega$ we have that
\begin{equation}
\begin{aligned}
      &\liminf_{\rho \to 0} \mathcal{A}_\rho(x,\phi, \phi(x) + \Lambda(\rho,x) ) \geq F^*(x, \phi(x), \nabla\phi(x), D^2\phi(x)), \\
      \big(\textup{resp.} \quad  &\limsup_{\rho \to 0} \mathcal{A}_\rho(x,\phi, \phi(x) + \Lambda(\rho,x) ) \leq F_*(x, \phi(x), \nabla\phi(x), D^2\phi(x))\big),
      \end{aligned}
\end{equation}         
    for any~$\Lambda$ such that~$|\Lambda|\leq \omega$ for some~$\omega$ as in~\Cref{def:weakviscasexo}. Then, if~$u$ is a viscosity supersolution (resp.\@ subsolution) of~\eqref{eq:defFsecas}, it is also a  viscosity supersolution associated to~$\omega$ (resp.\@ subsolution) as in~\Cref{def:weakviscasexo}.
\end{enumerate} 
\end{Theorem}
\begin{proof}
    The proof is straightforward and follows from the definitions.
\end{proof}

\section{Extensions} \label{se:extensions}

Let us show some natural extensions to the theory developed throughout the paper.

\subsection{Parabolic problems.} 
Our results can be trivially extended to cover parabolic equations of the form
$$u_t(\hat{x},t) + F(\hat{x},t,u(\hat{x},t),\nabla_x u(\hat{x},t), D^2u(\hat{x},t) )=0$$
for~$(\hat{x},t)\in\widehat{\Omega}\times (0,\infty)$,  with initial data~$u(\widehat{x},0)=u_0(\widehat{x})$ and boundary conditions~$u(\widehat{x},t)=\hat{g}(\widehat{x})$ for~$\hat{x}\in\partial\widehat{\Omega}$ and~$t\in (0,\infty)$. This lies directly in our framework by taking~$x\coloneqq(\hat{x},t)$ and~$\Omega\coloneqq\widehat{\Omega}\times (0,\infty)$, considering the boundary condition
\[
g(x)\coloneqq
\begin{cases}
    u_0(\hat{x}) \quad &\textup{if} \quad x\in \widehat{\Omega}\times\{0\},\\
    \hat{g}(\hat{x}) \quad &\textup{if} \quad x\in \partial \widehat{\Omega}\times(0,\infty),
\end{cases}
\]
and adopting the terminology~$\nabla u(x,t)=(\nabla_x u(x,t), \partial_t(x,t))$.
In this case we have to deal with equations of the form
$$u_\rho (\hat{x},t) = {\mathfrak{a}}_\rho (\hat{x},t,u_\rho)$$ in~$\widehat{\Omega} \times (0,\infty)$ with exterior values~$u_\rho (\widehat{x},t) = \hat{g} (\hat{x})$
in~$\RR^N\setminus \Omega = (\RR^N \setminus \widehat{\Omega}) \times (0,\infty)$ and initial condition~$u_\rho (x,t) = u_0 (x)$ for~$x\in \Omega$,~$t\leq 0$.
For previous references on mean value formulas for parabolic problems we refer to~\cite{[Blanc and Rossi 2019],KohnSerfatyCurvature,KohnSerfatyFullyNonlinear,AsymptoticMeanValueTugofwar}.

We will discuss an example of this type in~\Cref{se:examples}.

\subsection{Fractional/nonlocal problems. } We can also think of equations with a nonlocal term. More precisely, one can consider exterior boundary problems of the form
\begin{equation} \label{eq:BVPPDEnolocal} 
\begin{cases}
    F\left(x, u(x), \nabla u(x), D^2u(x), \mathcal{I}u(x) \right) = 0, \quad &x\in \Omega,\\
    u(x)=g(x),  \quad &x\in \RR^N\setminus \Omega,
\end{cases}
\end{equation}
where~$\mathcal{I}$ is nonlocal operator with the right monotonicity properties and~$F: \Omega\times \RR \times \RR^N \times S^{N} \times \RR$ is nondecreasing in the last variable. The class of admissible nonlocal operators is wide and we will discuss some examples in the next section. If we restrict ourselves to symmetric linear operators~$F=\mathcal{I}$, a natural class that one could consider is the class of the so-called Lévy operators, defined for a smooth function~$\phi$ by
\[
\mathcal{I}\phi(x)\coloneqq \textup{P.V.} \int_{|z|>0} (\phi(x)-\phi(x+z)) d\mu(z),
\]
where~$\mu$ is a nonnegative Radon measure in~$\RR^N\setminus\{0\}$ such that
\[
\int_{\RR^N} \min\{|z|^2,1\}d\mu(z)<+\infty.
\]
A simple consistent approximation of~$\mathcal{I}$ is given by
\[
\mathcal{A}_\rho(x,\phi,\phi(x))= \int_{|z|>f_\mu(\rho)} (\phi(x)-\phi(x+z)) d\mu(z),
\]
for some nondecreasing~$f_\mu$ such that~$f_\mu(0)=0$ and is strictly increasing in a neighborhood of~$0$. Note that once we have removed the set where~$\mu$ is singular, we can rewrite
\[
\mathcal{A}_\rho(x,\phi,\phi(x))=\mu(|z|> f_\mu(\rho)) \left( \phi(x) - \mathfrak{a}_\rho(x,\phi) \right) \quad \textup{with} \quad \mathfrak{a}_\rho(x,\phi)\coloneqq\dashint_{|z|> f_\mu(\rho)} \phi(x+z)\,d\mu(z).
\]
This lies precisely in the framework and notation of this paper by choosing~$f_\mu$ such that 
\[
\rho= \frac{1}{\mu(|z|> f_\mu(\rho))}.
\]

We could also consider nonlocal quasilinear operators, as we will discuss later in~\Cref{se:examples}.

\subsection{Spaces more general that~$\RR^N$}

As we already mentioned in~\Cref{obs:RN}, there is no reason to restrict ourselves to working on~$\RR^N$, since many of our results hold on more general spaces such as metric spaces, see~\cite{LiuSchikorra}. This choice was meant to keep the presentation simple. For sections~\ref{se:existencia} and~\ref{se:comparison} where no PDE is yet involved, the adaptation is trivial. 

\section{Examples} \label{se:examples}

In this section we will always assume that~$\mathcal{X}$ is the set of bounded measurable functions in~$\RR^N$.

 \subsection{The Laplacian.} 

 Our starting point is the boundary value problem
\begin{equation}\label{eq:NHLaplaceProb}
\begin{cases}
-\Delta u(x) - f(x) = 0, & x \in \Omega, \\
u(x) = g(x), & x \in \RR^N \setminus \Omega,
\end{cases}
\end{equation}
where~$f,g \in \overline{\mathcal{X}}$ and~$\Omega \subset \RR^N$ is an open, bounded domain.
It is a classical and standard computation that for~$\phi \in C^4_b(\RR^N)$ we have the following consistent approximation of the Laplacian, valid for all~$\rho \in (0,1)$,
\begin{equation}\label{eq:consStrongLap}
-\Delta \phi(x) - f(x)
= \frac{2(N+2)}{\rho}
\left(\phi(x) - \dashint_{B_{\sqrt{\rho}}(x)} \phi(y)\, dy\right)
- f(x) + \Lambda(x,\rho),
\end{equation}
with
\begin{equation}\label{eq:consitencyLap}
\sup_{x\in \Omega} |\Lambda(x,\rho)|
\leq c_{N} \|D^4 \phi\|_{L^\infty(\Omega_1)} \rho,
\end{equation}
where~$\Omega_1 \coloneqq \{x\in \RR^N : \textup{dist}(x,\Omega)<1\}$. In the notation of this paper, this corresponds to
\[
\mathcal{A}_{\rho}(x,\phi,\phi(x))
= \frac{2(N+2)}{\rho}
\left(\phi(x) - \dashint_{B_{\sqrt{\rho}}(x)} \phi(y)\, dy\right)
- f(x),
\]
and
\[
\mathfrak{a}_\rho(x,\phi)
= \dashint_{B_{\sqrt{\rho}}(x)} \phi(y)\, dy+\frac{\rho}{2(N+2)}f(x).
\]
Clearly,~$\mathcal{A}_\rho$ satisfies the structural assumptions
\textup{\ref{asA-it:a}--\ref{asA-it:abis}--\ref{asA-it:b}}.

To verify the existence of viscosity solutions of the DPP, we check hypotheses
\eqref{H1} and~\eqref{H2} in~\Cref{thm:perron}.
Choose~$R>0$ large enough so that~$\Omega_1 \subset B_{R/2}$ and define
\[
\overline{u}(x)
\coloneqq
\frac{\|f\|_{L^\infty(\Omega)}+1}{2N}(R^2-|x|^2)_+
+ \|g\|_{\infty},
\quad \textup{and} \quad 
\underline{u}(x)\coloneqq -\overline{u}(x).
\]
It is standard to check that~$\overline{u},\underline{u} \in C^4_b(\Omega_1)$, and that
\[
\begin{cases}
-\Delta \overline{u}(x)=\|f\|_{L^\infty(\Omega)}+1, & x \in \Omega, \\
\overline{u}(x) \geq \|g\|_{\infty}, & x \in \RR^N \setminus \Omega,
\end{cases}
\quad \textup{and} \quad 
\begin{cases}
-\Delta \underline{u}(x)=-\|f\|_{L^\infty(\Omega)}-1, & x \in \Omega, \\
\underline{u}(x) \leq -\|g\|_{\infty}, & x \in \RR^N \setminus \Omega.
\end{cases}
\]
By the consistency estimate~\eqref{eq:consitencyLap}, for~$\rho>0$ small enough we obtain
\[
\begin{cases}
\mathcal{A}_{\rho}(x,\overline{u},\overline{u}(x)) \geq \tfrac12, & x \in \Omega, \\
\overline{u}(x) \geq g(x), & x \in \RR^N \setminus \Omega,
\end{cases}
\quad \textup{and} \quad 
\begin{cases}
\mathcal{A}_{\rho}(x,\underline{u},\underline{u}(x)) \leq -\tfrac12, & x \in \Omega, \\
\underline{u}(x) \leq g(x), & x \in \RR^N \setminus \Omega.
\end{cases}
\]
Thus,~$\overline{u}$ (resp.\@~$\underline{u}$) is a classical strict supersolution (resp.\@ subsolution) of the DPP.
In particular, it is a viscosity supersolution (resp.\@ subsolution), and hence hypothesis
\eqref{H1e} (resp.\@~\eqref{H1}) holds.
Moreover, by the comparison result in~\Cref{thm:striccomparison}
(it is straightforward to check that~$\mathcal{X}$ and~$\mathcal{A}_\rho$ satisfy
assumptions~\eqref{visesta-item1}–\eqref{visesta-item2}–\eqref{visesta-item3}),
any viscosity subsolution (resp.\@ supersolution)~$v$ satisfies
$v\leq \overline{u}$ (resp.\@~$v\geq \underline{u}$),
which implies~\eqref{H2} (resp.\@~\eqref{H2e}).

Once existence of viscosity solutions is established, convergence of viscosity solutions of the DPP to the solution of~\eqref{eq:NHLaplaceProb} follows readily.
Indeed, the barriers~$\overline{u}$ and~$\underline{u}$ are bounded independently of~$\rho$, and therefore any viscosity solution of the DPP is uniformly bounded.
This yields the weak stability required in~\Cref{teo.convergencia.intro.bis}.
Furthermore, property~\eqref{eq:consitencyLap} implies the consistency conditions~\eqref{cons1}–\eqref{cons2} required in~\Cref{teo.convergencia.intro.bis}
(for instance with~$\omega(\rho)=\rho^2$)
and~\Cref{cor:convSchm}, ensuring convergence of the DPP under the hypothesis
of the strong uniqueness property
(assuming sufficient regularity of~$f$,~$g$ and~$\Omega$, cf.~\cite{BarlesSouganidis}).

Finally,~\Cref{teo.asympt.intro} also follows, since
\eqref{eq:consitencyLap} implies the consistency assumption therein,
which yields the equivalence between viscosity solutions of the PDE
and the asymptotic mean value expansion.

 \subsection{The Heat Equation.} Our starting point is the parabolic boundary value problem
\begin{equation}\label{eq:NHHeat}
\begin{cases}
\partial_t u - \Delta u - f(x,t) = 0, & (x,t) \in \Omega\times(0,T), \\
u(x,t) = g(x,t), & (x,t) \in (\RR^N \setminus \Omega)\times(0,T),\\
u(x,0) = u_0(x), & x \in  \Omega.
\end{cases}
\end{equation}
By a Taylor expansion, it follows that for~$\phi\in C^4_b(\RR^N\times(0,\infty))$ we have the following consistent approximation of the heat operator:
\begin{equation}\label{eq:consStrongHeatEq}
\partial_t \phi(x,t)-\Delta \phi(x,t)
= \frac{2(N+2)}{\rho}
\left(
\phi(x,t)
- \dashint_{t-\frac{\rho}{N+2}}^t
\dashint_{B_{\sqrt{\rho}}(x)} \phi(y,s)\, dy\, ds
\right)
- f(x,t) + \Lambda(x,t,\rho).
\end{equation}
Moreover,
\begin{equation}\label{eq:consHeat}
\sup_{(x,t)\in \Omega \times [0,T]}|\Lambda(x,t,\rho)|
\leq c_{N}
\left(
\|D^4 \phi\|_{L^\infty\left(\Omega_1\times[0,T+1]\right)}
+ \|\partial_{tt} \phi\|_{L^\infty\left(\Omega_1\times[0,T+1]\right)}
\right)\rho,
\end{equation}
where~$\Omega_1 \coloneqq \{x\in \RR^N : \textup{dist}(x,\Omega)<1\}$.
In the notation of this paper, this corresponds to
\[
\mathcal{A}_{\rho}(x,t,\phi,\phi(x,t))
= \frac{2(N+2)}{\rho}
\left(
\phi(x,t)
- \dashint_{t-\frac{\rho}{N+2}}^t
\dashint_{B_{\sqrt{\rho}}(x)} \phi(y,s)\, dy\, ds
\right)
- f(x,t),
\]
and
\[
\mathfrak{a}_\rho(x,t,\phi)
= \dashint_{t-\frac{\rho}{N+2}}^t
\dashint_{B_{\sqrt{\rho}}(x)} \phi(y,s)\, dy\, ds
+ \frac{\rho}{2(N+2)}f(x,t).
\]

Let us now observe that, given a smooth and bounded nonnegative function~$\psi$ such that
\[
\partial_t \psi - \Delta \psi
= \|f\|_{L^\infty(\Omega\times(0,T))}+1
\quad \textup{for all} \quad (x,t) \in B_R\times(0,T+1),
\] the function~$\overline{u}
\coloneqq
\psi
+ \|g\|_{L^\infty(\Omega\times(0,T))}
+ \|u_0\|_{L^\infty(\Omega)}
$
satisfies
\begin{equation}
\begin{cases}
\partial_t \overline{u} - \Delta \overline{u}
= \|f\|_{L^\infty(\Omega\times(0,T))}+1,
& (x,t) \in \Omega\times(0,T), \\
\overline{u}(x,t)
\geq \|g\|_{L^\infty(\Omega\times(0,T))},
& (x,t) \in (\RR^N \setminus \Omega)\times[0,T),\\
\overline{u}(x,0)
\geq \|u_0\|_{L^\infty(\Omega)},
& x \in \Omega,
\end{cases}
\end{equation}
provided that~$\Omega_1 \subset B_{R/2}$.
Moreover,~$\overline{u}$ is smooth enough in~$\Omega_1\times[0,T+1]$ so that
the consistency estimate~\eqref{eq:consHeat} applies.

From this point on, conclusions analogous to those obtained in the elliptic case
can be carried out for the parabolic setting.

\subsection{$p$-harmonic functions.}
Our starting point is the equation, for~$p>1$,
\begin{equation}\label{eq:pLaplaceProb}
\begin{cases}
-\Delta_p u(x) = 0, & x \in \Omega, \\
u(x) = g(x), & x \in \RR^N \setminus \Omega.
\end{cases}
\end{equation}
In this case, we first need to point out the equivalence between viscosity solutions of the above problem and those of
\[
\begin{cases}
-\Delta_p^{\textup{N}} u(x) = 0, & x \in \Omega, \\
u(x) = g(x), & x \in \RR^N \setminus \Omega,
\end{cases}
\]
where
\[
\Delta_p^{\textup{N}} u \coloneqq |\nabla u|^{2-p}\Delta_p u
= \Delta u + (p-2)\Bigl\langle D^2 u \, \frac{\nabla u}{|\nabla u|}, \frac{\nabla u}{|\nabla u|} \Bigr\rangle.
\]

A consistent approximation was found in~\cite{AsymptoticMeanValuePharmonic} and is given, for all~$\phi \in C_b^4(\RR^N)$ such that
$\nabla \phi \neq 0$ in~$\Omega_1$, by
\[
\Delta_p^{\textup{N}} \phi(x)
= \underbrace{\frac{1}{\rho}\left(\phi(x)-\mathfrak{a}_\rho(x,\phi)\right)}_{\mathcal{A}_\rho(x,\phi,\phi(x))}
+ \Lambda(x,\rho),
\]
where
\[
\mathfrak{a}_\rho(x,\phi)
\coloneqq
\frac{p-2}{p+N}
\left(
\frac{1}{2}\sup_{B_{\gamma \sqrt{\rho}}(x)} \phi
+ \frac{1}{2}\inf_{B_{\gamma \sqrt{\rho}}(x)} \phi
\right)
+ \frac{N+2}{N+p}\dashint_{B_{\gamma \sqrt{\rho}}(x)} \phi(y)\,dy,
\]
and
\begin{equation}\label{eq:consitencypLap}
\sup_{x\in \Omega} |\Lambda(x,\rho)|
\leq c_{N,p}
\frac{\|\phi\|_{C^4(\Omega_1)}}{\inf_{x\in \Omega_1}|\nabla \phi(x)|}\,\rho^{1/2}.
\end{equation}

Note that in this case structural property~\textup{\ref{asA-it:a}} only holds for~$p \geq 2$, but analogous ideas can be found in the literature for~$p \in (1,2)$.

\Cref{rem:restrictedTest} is particularly important in this setting: although consistency only holds for test functions with nonvanishing gradient, in the viscosity framework one only needs to test against such functions (see~\cite{AsymptoticMeanValuePharmonic}).

To construct the barriers, let~$x_0 \notin \Omega_1$ and ~$R>0$ be so that
\[
\Omega_1 \subset B_R(x_0)
\quad \textup{and} \quad
\Omega_1 \cap B_{R/4}(x_0) = \emptyset.
\]
We then consider a function~$\overline{u}$ satisfying
\begin{equation}
\begin{cases}
-\Delta_p^N \overline{u}(x) \geq 1, & x \in B_R(x_0), \\
\overline{u}(x) \geq g(x), & x \in \RR^N\setminus \Omega.
\end{cases}
\end{equation}
An explicit example, given for instance in~\cite{delTesoLindgrenMeanValue}, is
\[
\overline{u}(x)=C_1|x_0-x|^{\frac{p}{p-1}}+C_2,
\]
for appropriate choices of the constants~$C_1$ and~$C_2$, depending on~$R$,~$\Omega$,~$p$, and~$N$.
Note that~$\overline{u}$ is smooth, has nonvanishing gradient in~$\Omega_1 \subset \RR^N \setminus B_{R/4}(x_0)$, and is bounded in~$\Omega_1$.

From here, consequences analogous to those obtained in the harmonic case can be derived. Let us finally observe that the strong uniqueness property was proved in~\cite{DelTeso-Manfredi.Parviainen}.

\subsection{The~$p$-Laplacian.}
The idea used in the previous example only works for the homogeneous problem. Let us consider the non-homogeneous problem
\begin{equation}\label{eq:pLaplaceProb2}
\begin{cases}
-\Delta_p u(x) - f(x)=0, & x \in \Omega, \\
u(x) = g(x), & x \in \RR^N \setminus \Omega.
\end{cases}
\end{equation}
It was shown in~\cite{delTesoLindgrenMeanValue, delTesoMedinaOchoa2024} that the following consistent approximation holds for all
$\phi \in C^4(B_R(x_0))$:
\begin{equation}
-\Delta_p \phi(x)-f(x)
= \mathcal{A}_\rho(x,\phi,\phi(x)) + \Lambda(x,\rho),
\end{equation}
where
\begin{equation}
\mathcal{A}_\rho(x,\phi,\phi(x))
= \frac{C}{\rho}
\dashint_{B_{\rho^{\frac{1}{p}}}}
|\phi(x)-\phi(x+y)|^{p-2}
\bigl(\phi(x)-\phi(x+y)\bigr)\, dy-f(x)
\end{equation}
and
\begin{equation}\label{eq:consitencypLap2}
\sup_{x\in \Omega} |\Lambda(x,\rho)|
\leq c_{N,p} \|\phi\|_{C^4(\Omega_1)} \rho^{\frac{\alpha}{p}},
\end{equation}
with~$\alpha=2$ if~$p\geq4$ and~$\alpha=p-2$ if~$p\in(2,4)$, and where
$\Omega_1 \coloneqq \{x\in \RR^N : \textup{dist}(x,\Omega)<1\}$.

Using the same barrier functions as in the previous section, one can obtain results analogous to those derived in the homogeneous case. Let us finally observe that the strong uniqueness property was proved in~\cite{delTesoLindgrenMeanValue}.

Note that, in this setting, there is no explicit formulation for the associated averaging operator~$\mathfrak{a}_\rho$.

\subsection{The fractional Laplacian.} Our starting point is the equation
\begin{equation}
\begin{cases}
(-\Delta)^s u(x) -f(x)=0, & x \in \Omega, \\
u(x) = g(x), & x \in \RR^N \setminus \Omega,
\end{cases}
\end{equation}
where~$s \in (0,1)$ and
\[
(-\Delta)^s \phi(x)
= c_{N,s}\,\textup{P.V.} \int_{\RR^N} \bigl(\phi(x)-\phi(x+z)\bigr)\frac{dz}{|z|^{N+2s}},
\]
with~$c_{N,s}$ a positive normalization constant.
It is standard to check that for~$\phi \in C_b^2(\RR^N)$ the following consistency relation holds:
\[
(-\Delta)^s \phi(x)-f(x)
= c_{N,s} \int_{|z|>\varepsilon} \bigl(\phi(x)-\phi(x+z)\bigr)\frac{dz}{|z|^{N+2s}}-f(x)
+ \Lambda(x,\varepsilon),
\]
with
\begin{equation}\label{eq:consfracLap}
\sup_{x\in \Omega} |\Lambda(x,\varepsilon)|
\leq k_{N,s}\|D^2 \phi\|_{L^\infty(\Omega_1)}
\varepsilon^{2-2s},
\end{equation}
where~$\Omega_1 \coloneqq \{x\in \RR^N : \textup{dist}(x,\Omega)<1\}$.

Introducing the measure~$d\mu_s(z)=|z|^{-N-2s}\,dz$ and setting~$\rho=\varepsilon^{2s}$, we obtain, in the notation of this paper,
\[
\mathcal{A}_\rho(x,\phi,\phi(x))
= \frac{|\partial B_1|\,c_{N,s}}{2s} \frac{1}{\rho}
\left(\phi(x)-\dashint_{|z|>\rho^{\frac{1}{2s}}} \phi(x+z)\,d\mu_s(z)\right) - f(x),
\]
and
\[
\mathfrak{a}_\rho(x,\phi)
= \dashint_{|z|>\rho^{\frac{1}{2s}}} \phi(x+z)\,d\mu_s(z)+f(x).
\]

As usual, we construct a barrier for the DPP by choosing~$R>0$ large enough so that
$\Omega_1 \subset B_{R/2}$ and defining
\begin{equation}\label{eq:FracLapequal1}
\overline{u}(x)
\coloneqq
C_{N,s}\bigl(\|f\|_{L^\infty(\Omega)}+1\bigr)(R^2-|x|^2)_+^{\,s}
+ \|g\|_{L^\infty(\RR^N)},
\quad \textup{with} \quad
C_{N,s}
= \frac{2^{-2s}\Gamma\!\left(\frac{N}{2}\right)}
{\Gamma\!\left(\frac{N+2s}{2}\right)\Gamma(1+s)}.
\end{equation}
This function satisfies
\[
\begin{cases}
(-\Delta)^s \overline{u}(x)=\|f\|_{L^\infty(\Omega)}+1, & x \in \Omega, \\
\overline{u}(x) \geq \|g\|_{L^\infty(\RR^N)}, & x \in \RR^N \setminus \Omega.
\end{cases}
\]

From here, using the consistency estimate~\eqref{eq:consfracLap}, conclusions analogous to those obtained in the harmonic case can be derived.
\subsection{The infinity fractional Laplacian.} \label{ej:inflapl}
We develop this example in a more thoughtful way, since the existence result for solutions of the DPP is new and was previously unknown even in the classical sense due to measurability issues.

Our starting point is the equation
\begin{equation}\label{eq:pfracLapinf}
\begin{cases}
-\Delta^s_\infty u(x) - f(x) = 0, & x \in \Omega, \\
u(x) = g(x), & x \in \RR^N \setminus \Omega,
\end{cases}
\end{equation}
where the infinity fractional Laplacian~$\Delta^s_\infty$, introduced in~\cite{[BCFinfinity]}, is defined for~$s \in (1/2,1)$ as
\[
-\Delta^s_\infty \varphi(x)
:= C_s \sup_{|y|=1} \inf_{|\tilde y|=1}
\int_0^{\infty}
\left(2\varphi(x)-\varphi(x+\eta y)-\varphi(x-\eta \tilde y)\right)
\frac{d\eta}{\eta^{1+2s}},
\]
for a positive constant~$C_s$ that coincides with the one-dimensional constant of the fractional Laplacian.

The following consistent approximation of this operator was obtained in~\cite{delTesoEndalLewicka}
\[
-\Delta_\infty^s \phi(x)-f(x)
= \mathcal{A}_\rho(x,\phi,\phi(x)) + \Lambda(x,\rho),
\]
with
\[
\mathcal{A}_\rho(x,\phi,\phi(x))
= \frac{C_s}{s}\frac{1}{\rho}
\left(
\phi(x)
-\frac{1}{2}
\sup_{|y|=1}
\dashint_{\rho^{\frac{1}{2s}}}^{\infty}
\phi(x+\eta y)\,\frac{d\eta}{\eta^{1+2s}}
-\frac{1}{2}
\inf_{|y|=1}
\dashint_{\rho^{\frac{1}{2s}}}^{\infty}
\phi(x+\eta y)\,\frac{d\eta}{\eta^{1+2s}}
\right)
-f(x),
\]
and
\begin{equation}
\sup_{x\in \Omega} |\Lambda(x,\rho)|
\leq c_{N,s}
\frac{\|\phi\|_{C^2(\Omega_1)}}{\inf_{x\in \Omega_1}|\nabla \phi(x)|}\,
\rho^{\alpha},
\end{equation}
for some~$\alpha>0$ depending on~$N$ and~$s$.
Note that, as in the normalized~$p$-Laplacian case, consistency only holds when the gradient of~$\phi$ does not vanish.

In order to prove existence of classical solutions of the DPP
\begin{equation}\label{eq:pfracLapinfDPP}
\begin{cases}
\mathcal{A}_\rho(x,u_\rho,u_\rho(x)) = 0, & x \in \Omega, \\
u_\rho(x) = g(x), & x \in \RR^N \setminus \Omega,
\end{cases}
\end{equation}
a major measurability issue arises. The function~$u_\rho$ must be measurable along the one-dimensional rays
$x+\eta y$ with~$\eta \in (\rho^{\frac{1}{2s}},\infty)$, for all~$x \in \Omega$ and all~$y \in \partial B_1$.
This is, in general, difficult to verify and was not established in~\cite{delTesoEndalLewicka}.
However, within the framework developed in this paper, we are able to prove existence of viscosity solutions of the DPP.

To this end, we verify that hypotheses~\eqref{H1} and~\eqref{H2} hold.
We construct classical strict super- and subsolutions of the DPP.
Choose~$R>0$ large enough so that~$\Omega_1 \subset B_{R/2}$ and define
\[
\overline{u}(x)
\coloneqq
k_{s}\bigl(\|f\|_{L^\infty(\Omega)}+1\bigr)(R^2-|x|^2)_+^{\,s}
+ \|g\|_{L^\infty(\RR^N)},
\quad \textup{with} \quad
k_{s}
= \frac{2^{-2s}\Gamma\!\left(\frac{1}{2}\right)}
{\Gamma\!\left(\frac{1+2s}{2}\right)\Gamma(1+s)}.
\]
This is a radial function whose one-dimensional profile
$\overline{U}:\RR\to \RR_+$ coincides with~\eqref{eq:FracLapequal1}, namely
\[
\overline{U}(r)
= k_{s}\bigl(\|f\|_{L^\infty(\Omega)}+1\bigr)(R^2-r^2)_+^{\,s}
+ \|g\|_{L^\infty(\RR^N)},
\]
and~$\overline{u}(x)=\overline{U}(|x|)$ for all~$x\in \RR^N$.
Since~$\overline{u}$ is radial and radially nonincreasing, Proposition~6.1 in~\cite{delTesoEndalJakobsenVazquez} yields
\[
-\Delta_\infty^s \overline{u}(x)
= (-\partial_{rr}^2)^s \overline{U}(|x|),
\]
where~$(-\partial_{rr}^2)^s$ denotes the one-dimensional fractional Laplacian.
Thus,
\[
\begin{cases}
-\Delta^s_\infty \overline{u}(x)
= \|f\|_{L^\infty(\Omega)}+1, & x \in \Omega, \\
\overline{u}(x)
\geq \|g\|_{L^\infty(\RR^N)}, & x \in \RR^N \setminus \Omega.
\end{cases}
\]

By a similar radial argument, we compute
\begin{align}
\mathcal{A}_\rho(x,\overline{u},\overline{u}(x))
&=
\frac{C_s}{s}\frac{1}{\rho}
\left(
\overline{U}(r)
- \frac 12 \dashint_{\rho^{\frac{1}{2s}}}^{\infty}
\overline{U}(r+\eta)\,\frac{d\eta}{\eta^{1+2s}}- \frac 12 \dashint_{\rho^{\frac{1}{2s}}}^{\infty}
\overline{U}(r-\eta)\,\frac{d\eta}{\eta^{1+2s}}
\right)
-f(x) \\
&\eqqcolon
\widetilde{\mathcal{A}}_\rho(r,\overline{U},\overline{U}(r)) - f(x),
\end{align}
with~$r=|x|$.
Proceeding as in the previous section, we obtain
 \begin{align}
 \sup_{x\in \Omega} |-\Delta_\infty^s \overline{u}(x)-f(x)-\mathcal{A}_\rho(x,\overline{u},\overline{u}(x))|&= \sup_{x\in \Omega} |(-\partial_{rr}^2)^s \overline{U}(|x|)-\widetilde{\mathcal{A}}_\rho(|x|,\overline{U},\overline{U}(|x|))|\\
 &\leq \tilde{c}_{s}\|D^2 \overline{u}\|_{L^\infty(\Omega_1)}
 \rho^{\frac{2-2s}{2s}}.
 \end{align}
Choosing~$\rho>0$ sufficiently small, we conclude that
\[
\begin{cases}
\mathcal{A}_\rho(x,\overline{u},\overline{u}(x)) \geq \frac{1}{2},
& x \in \Omega, \\
\overline{u}(x) \geq g(x),
& x \in \RR^N \setminus \Omega.
\end{cases}
\]
Hence, by~\Cref{thm:striccomparison}, hypothesis~\eqref{H2} holds.
Moreover,~\eqref{H1} follows since
$\underline{u}\coloneqq -\overline{u}$
is a classical (and therefore viscosity) subsolution of the DPP.

Finally, we remark that a strong uniqueness property, or even a comparison principle for viscosity solutions, is unknown for this PDE problem, see~\cite{[BCFinfinity]}.
As a consequence, full convergence of the DPP cannot be achieved.
Nevertheless, the uniform barriers constructed above provide the weak stability required for
\Cref{teo.convergencia.intro.bis}, ensuring the existence of viscosity subsolutions and supersolutions. As in the~$p$-harmonic case, the notion of viscosity solution must be restricted to test functions with nonvanishing gradient, since the consistency result proved in~\cite{delTesoEndalLewicka} only holds in this class. Within this framework, the equivalence between viscosity solutions of the asymptotic expansion and of the PDE also follows from~\Cref{thm:weakvisctivisc}.

\section*{Acknowledgements}

This research started during a visit of \textsc{F. del Teso} and \textsc{J. Ruiz-Cases} to Universidad de Buenos Aires, and was finished during a visit of \textsc{J.\,D.\,Rossi} to Universidad Autónoma de Madrid. The authors want to thank both institutions for the hospitality and the stimulating working atmosphere. \textsc{F. del Teso} is funded by the Ramón y Cajal contract reference RYC2020-029589-I and the
research projects PID2021-127105NB-I00, CNS2024-154515 of the AEI, Government
of Spain. \textsc{J.\,D.\,Rossi} was also supported by CONICET grants PIP GI No 11220150100036CO, PICT-2018-03183 and UBACyT grant 20020160100155BA, Argentina. \textsc{J. Ruiz-Cases} was supported by grants PID2020-116949GB-I00, PID2023-146931NB-I00, RED2022-134784-T, RED2024-153842-T and CEX2023-001347-S, all of them funded by MICIU/AEI.

\bibliographystyle{acm}

\end{document}